\documentclass[a4paper,10pt]{amsart}
\usepackage{amsmath,amssymb,amsthm}
\usepackage[T1]{fontenc}
\usepackage[all,cmtip]{xy}
\usepackage[dvipsnames]{xcolor}
\usepackage[linktocpage=true,colorlinks=true,linkcolor=red,citecolor=blue,urlcolor=green]{hyperref}
\usepackage[capitalise]{cleveref}
\usepackage[normalem]{ulem}
\usepackage[shortlabels]{enumitem}
\setlist{font=\normalfont\upshape}
\numberwithin{equation}{section}
\crefname{equation}{}{}
\let\newterm\emph\relax
\newtheorem{theorem}{Theorem}[section]
\newtheorem{prop}[theorem]{Proposition}
\newtheorem{cor}[theorem]{Corollary}
\newtheorem{lemma}[theorem]{Lemma}

\theoremstyle{definition}

\newtheorem{rmk}[theorem]{Remark}
\newtheorem{ex}[theorem]{Example}

\DeclareMathOperator{\Hom}{Hom}
\DeclareMathOperator{\End}{End}
\DeclareMathOperator{\dgEnd}{dgEnd}
\DeclareMathOperator{\Ext}{Ext}
\DeclareMathOperator{\Tor}{Tor}
\DeclareMathOperator{\RHom}{\mathbf{R}Hom}
\DeclareMathOperator{\RGamma}{\mathbf{R}\Gamma}
\let\ctrtensor\odot\relax

\newcommand{\Acal}{\mathcal{A}}
\newcommand{\Bcal}{\mathcal{B}}
\newcommand{\Ccal}{\mathcal{C}}

\newcommand{\Hcal}{\mathcal{H}}

\newcommand{\Tcal}{\mathcal{T}}
\newcommand{\Ucal}{\mathcal{U}}
\newcommand{\Vcal}{\mathcal{V}}

\newcommand{\Qbb}{\mathbb{Q}}

\newcommand{\Zbb}{\mathbb{Z}}

\newcommand{\pp}{\mathfrak{p}}
\newcommand{\qq}{\mathfrak{q}}
\newcommand{\mm}{\mathfrak{m}}
\newcommand{\Sfr}{\mathfrak{S}}
\newcommand{\Mfr}{\mathfrak{M}}
\newcommand{\ef}{\mathsf{f}}
\newcommand{\ed}{\mathsf{d}}

\newcommand{\D}{\mathcal{D}}
\newcommand{\K}{\mathcal{K}}
\newcommand{\fp}[1]{\mathrm{fp}(#1)}
\newcommand{\rmod}{\mathrm{mod}\mathchar`-}\let\mod\rmod\relax 
\newcommand{\rMod}{\mathrm{Mod}\mathchar`-}\let\Mod\rMod
\newcommand{\lMod}{\mathchar`-\mathrm{Mod}}

\newcommand{\lDisc}{\mathchar`-\mathrm{Disc}}\let\Disc\Disc\relax
\newcommand{\rdgMod}{\mathrm{dgMod}\mathchar`-}\let\dgMod\rdgMod\relax
\newcommand{\ldgMod}{\mathchar`-\mathrm{dgMod}}
\relax
\newcommand{\lInj}{\mathchar`-\mathrm{Inj}}
\let\Proj\rProj\relax
\relax

\let\Ctra\rCtra\relax
\newcommand{\Ctrapr}[1]{\mathrm{Ctra}_{\mathrm{Proj}}\mathchar`-#1}

\newcommand{\cpt}{\mathrm{c}}
\newcommand{\bdd}{\mathrm{b}}
\newcommand{\fg}{\mathrm{fg}}

\newcommand{\op}{\mathrm{op}}

\newcommand{\id}{\mathrm{injdim}}
\newcommand{\Spec}[1]{\mathrm{Spec}(#1)}
\DeclareMathOperator{\depth}{depth}
\DeclareMathOperator{\height}{ht}
\newcommand{\CM}[1]{\mathrm{CM}(#1)}
\DeclareMathOperator{\Supp}{Supp}

\newcommand*{\Perp}[1]{{}^{\bot_{#1}}}
\DeclareMathOperator{\real}{real}

\DeclareMathOperator{\Add}{Add}
\DeclareMathOperator{\Prod}{Prod}
\DeclareMathOperator{\Def}{Def}
\DeclareMathOperator{\thick}{thick}

\DeclareMathOperator{\Lex}{Lex}
\newcommand{\yo}{\mathbf{y}}
\newcommand{\toeq}{\mathrel{\vbox{\offinterlineskip\ialign{\hfil$##$\hfil\cr%
	\scriptstyle\cong\cr\to\cr}}}}
\def\leqdef{\mathrel{\mathrel{\mathop:}=}}

\newcommand{\wbar}[2][2]{%
	{}\mkern#1mu\overline{\mkern-#1mu#2}}
\newcommand{\what}[2][2]{%
	{}\mkern#1mu\widehat{\mkern-#1mu#2}}
\newcommand{\wtilde}[2][2]{%
	{}\mkern#1mu\widetilde{\mkern-#1mu#2}}
\let\rho\varrho\let\Gamma\varGamma\let\Sigma\varSigma\let\Phi\varPhi\relax

\title{Product-complete tilting complexes and Cohen--Macaulay hearts}
\author{Michal Hrbek}
\address[M.~Hrbek]{Institute of Mathematics of the Czech Academy of Sciences, \v{Z}itn\'{a}~25, 115 67~Prague, Czech Republic}
\email{hrbek@math.cas.cz}

\author{Lorenzo Martini}
\address[L.~Martini]{Dipartimento di Informatica -- Settore di Matematica, Universit\`a di Verona, Strada le Grazie~15 -- Ca’ Vignal, 37134~Verona, Italy}
\email{lorenzo.martini@univr.it}


\thanks{The first author was supported by the GA\v CR project 23-05148S and RVO:~67985840.}
\begin{document}
\begin{abstract}
We show that the cotilting heart associated to a tilting complex $T$ is a locally coherent and locally coperfect Grothendieck category (i.e.\ an  Ind-completion of a small artinian abelian category) if and only if $T$ is product-complete. We then apply this to the specific setting of the derived category of a commutative noetherian ring $R$. If $\dim(R)<\infty$, we show that there is a derived duality $\D^\bdd_\fg(R) \cong \D^\bdd(\Bcal)^\op$ between $\mod R$ and a noetherian abelian category $\Bcal$ if and only if $R$ is a homomorphic image of a Cohen--Macaulay ring. Along the way, we obtain new insights about t-structures in $\D^\bdd_\fg(R)$. In the final part, we apply our results to obtain a new characterization of the class of those finite-dimensional Noetherian rings that admit a Gorenstein complex.
\end{abstract}
\maketitle

\tableofcontents

\section{Introduction}%
As conjectured by Sharp \cite{Sha79} and proved almost twenty-five years later by Kawasaki \cite{Kaw02}, a commutative noetherian ring $R$ of finite Krull dimension admits a dualizing complex if and only if it is a homomorphic image of a Gorenstein ring. The existence of a dualizing complex amounts to the existence of a duality on the bounded derived category $\D^\bdd_\fg(R)$ of finitely generated $R$-modules, that is, to a triangle equivalence $\D^\bdd_\fg(R) \cong \D^\bdd_\fg(R)^\op$. In this paper, we characterize the existence of a more general and less symmetric form of duality. Namely, we show in \cref{char-eq} that there is a triangle equivalence $\D^\bdd_\fg(R) \cong \D^\bdd(\Bcal)^\op$ for some noetherian abelian category $\Bcal$ if and only if $R$ is a homomorphic image of a Cohen--Macaulay ring. Due to another deep result of Kawasaki \cite{Kaw08}, such rings are precisely the CM-excellent rings whose Zariski spectrum $\Spec R$ admits a codimension function (see \v Cesnavi\v cius \cite{Ces21} and Takahashi \cite{Tak22} for recent development of CM-excellent rings and schemes).

The main tool we use to obtain the result is the large tilting theory. The way tilting theory enters the picture can be explained already in the dualizing complex setting. Indeed, the duality $\D^\bdd_\fg(R) \cong \D^\bdd_\fg(R)^\op$ is actually a derived equivalence in hiding --- the natural isomorphism $\D^\bdd_\fg(R)^\op \simeq \D^\bdd((\mod R)^\op)$ yields a derived equivalence between the noetherian category $\mod R$ and the artinian category $(\mod R)^\op$. Following Yekutieli and Zhang \cite{YZ05} (cf.\ \cite[Theorem~6.2]{YZ06} for an extension to the non-affine case), the duality realizes $(\mod R)^\op$, up to equivalence, as the heart of the perverse t-structure in $\D^\bdd_\fg(R)$, which is obtained by dualizing the canonical t-structure. Alonso, Jerem\'\i as and Saor\'\i n \cite[\S 6]{ATJLS10} showed that the perverse t-structure extends to a compactly generated t-structure in $\D(R)$, which they call the Cohen--Macaulay t-structure. By a recent result of Pavon and Vit\'oria \cite{PV20}, this t-structure is cotilting and thus induces an unbounded derived equivalence $\D(\Hcal_\mathrm{CM}) \cong \D(R)$ between the heart $\Hcal_\mathrm{CM}$ and $\Mod R$.

Recently, silting and cosilting complexes induced by codimension functions, and the corresponding t-structures, have been studied and explicitly constructed by Nakamura, Šťovíček, and the first author in \cite{HNS}. In particular, the Cohen--Macaulay heart $\Hcal_\mathrm{CM}$ can be defined even in the absence of a dualizing complex; the caveat is that the question of when $\Hcal_\mathrm{CM}$ is derived equivalent to $\Mod R$ remains open in this generality, see \cite[Question~7.8]{HNS}. Now the derived equivalence behind the generalized duality $\D^\bdd_\fg(R) \cong \D^\bdd(\Bcal)^\op$ is between $\mod R$ and the artinian category $\Acal = \Bcal^\op$. It follows from the theory of Roos \cite{Roos69} (see \cref{Roos}) that the induced derived equivalence of unbounded derived categories is between the module category $\Mod R$ and a locally coherent and locally coperfect Grothendieck category (see \cref{ss:loccoh}), which turns out to be precisely $\Hcal_\mathrm{CM}$. 

We start in \cref{s:tilting} by recalling the basic notions of large tilting and cotilting theory, including a neat characterization \cref{tilting-bounded} of the derived equivalences induced by (bounded) tilting and cotilting complexes. Using the theory of topological endomorphism rings of tilting modules developed by Positselski and Šťovíček \cite{PS19b} and recently extended to tilting and cotilting complexes in \cite{Hrb22}, relying heavily on Positselski's theory of contramodules over topological rings \cite{contramodules}, we show in \cref{s:prod-comp} that the heart of a cotilting t-structure induced by a large tilting complex $T$ (over any associative ring $R$) is locally coherent and locally coperfect if and only if $T$ is product-complete. Starting with \cref{s:comnoeth}, we specialize to the setting of a commutative noetherian ring $R$ and recall the relevant aspects of the theory of compactly generated and restrictable t-structures in $\D(R)$. In \cref{prod-add}, we show that any codimension function induces a product-complete tilting complex if $R$ is a homomorphic image of a finite-dimensional Cohen--Macaulay ring. In the next \cref{s:cmexc}, we characterize the restrictability of a codimension filtration t-structure via the CM-excellent condition and show a sort of a converse to the recent result on t-structures in $\D^\bdd_\fg(R)$ of Takahashi \cite{Tak22}. In \cref{s:cmexc2}, we prove the promised characterization of homomorphic images of Cohen--Macaulay rings in terms of derived equivalences (see \cref{char-eq}). In the final \cref{s:gor} we apply our results to the theory of Gorenstein complexes. In particular, in \cref{Gorenstein-fpinj} and \cref{Gor-eq} we characterize finite-dimensional rings admitting a Gorenstein complex as those homomorphic images of Cohen--Macaulay rings for which $\fp{\Hcal_\mathrm{CM}}$, the abelian category of finitely presentable objects in the Cohen--Macaulay heart, admits an injective cogenerator (see \cref{fp-inj-cogen}).

\subsection*{Acknowledgements}
We are grateful to Leonid Positselski and Amnon Yekutieli for useful comments on an earlier draft of this paper. We would like to express our thanks to the anonymous referees for many very useful comments.

\section{Tilting and cotilting complexes}%
\label{s:tilting}%
Let $\Tcal$ be a triangulated category. A \newterm{t-structure} in $\Tcal$ is a pair $(\Ucal,\Vcal)$ of full subcategories such that $\Hom_\Tcal(\Ucal,\Vcal)=0$, $\Ucal[1]\subseteq\Ucal$, and such that for each $X \in \Tcal$ there is a triangle $U\buildrel u_X\over\to X \buildrel v_X\over\to V\buildrel+\over\to$ with $U \in \Ucal$ and $V \in \Vcal$. In fact, the latter triangle is functorially unique in the sense that $u_X$ is a $\Ucal$-coreflection map and $v_X$ is a $\Vcal$-reflection map of $X$ in $\Tcal$. The \newterm{heart} $\Hcal = \Ucal[-1] \cap \Vcal$ of the t-structure is an abelian category whose exact sequences are precisely the triangles of $\Tcal$ with all components belonging to $\Hcal$. See \cite{BBD81} for details. For any $Y \in \Tcal$ we define the following (always full, additive, isomorphism-closed) subcategories of $\Tcal$ by orthogonality relations:
\begin{align*}
	Y\Perp{>0} &= \{X \in \Tcal \mid \Hom_\Tcal(Y,X[i]) = 0\ \forall i>0\}, \cr
	Y\Perp{\leq 0} &= \{X \in \Tcal \mid \Hom_\Tcal(Y,X[i]) = 0\ \forall i\leq 0\}, \cr
	\Perp{>0}Y &= \{X \in \Tcal \mid \Hom_\Tcal(X,Y[i]) = 0\ \forall i>0\}, \cr
	\Perp{\leq 0}Y &= \{X \in \Tcal \mid \Hom_\Tcal(X,Y[i]) = 0\ \forall i\leq 0\}.
\end{align*}

Following Psaroudakis and Vit\'oria \cite{PV18} and Nicol\'as, Saor\'\i n, and Zvonareva \cite{NSZ19}, an object $T \in \Tcal$ is \newterm{silting} if $(T\Perp{>0},T\Perp{\leq 0})$ is a t-structure in $\Tcal$. We call the latter t-structure a \newterm{silting t-structure} induced by $T$, and we call two silting objects $T$ and $T'$ \newterm{equivalent} if they induce the same silting t-structure. Given an object $X$, let $\Add(X)$ be the subcategory of all direct summands of all coproducts of copies of $X$ which exist in $\Tcal$. If $\Tcal$ has all set-indexed coproducts then two silting objects $T$ and $T'$ are equivalent if and only if $\Add(T) = \Add(T')$ by \cite[Lemma 4.5(ii)]{PV18}. Let $\Hcal_T$ be the heart of the silting t-structure induced by $T$. We say that a silting object $T$ is \newterm{tilting} if $\Add(T) \subseteq \Hcal_T$. \newterm{Cosilting} and \newterm{cotilting} objects and t-structures are defined dually: an object $C$ is cosilting if $(\Perp{\leq 0}C,\Perp{>0}C)$ is a t-structure in $\Tcal$, and it is cotilting if in addition $\Prod(C) \subseteq \Hcal_C$, where $\Hcal_C$ is the heart of the t-structure and $\Prod(X)$ is the subcategory of all direct summands of arbitrary existing products of copies of $X$. Again we say that two cosilting objects $C$ and $C'$ are \newterm{equivalent} if they induce the same t-structure, and this happens precisely when $\Prod(C) = \Prod(C')$ provided that $\Tcal$ has all set-indexed products.

If $\Ccal$ is an abelian category, we let $\D(\Ccal)$ be the unbounded derived category and $\D^\bdd(\Ccal)$ the bounded derived category of cochain complexes. In all situations we consider, $\Ccal$ is either essentially small or it is Grothendieck, so the respective derived categories suffer no set-theoretic existential crises. Let $R$ be a (unital, associative) ring. We will be mostly interested in the case when the role of $\Tcal$ is played by one of $\D(\rMod R)$, $\D^\bdd(\rMod R)$, or $\D^\bdd(\mod R)$, the unbounded derived category of all right $R$-modules, the bounded derived category of all right $R$-modules, or the bounded derived category of finitely presented right $R$-modules, respectively; the last category is well-defined if and only if $R$ is a right coherent ring, which amounts to $\mod R$ being an abelian category. 

Recall that if $\Tcal = \D(\Ccal)$ or $\Tcal = \D^\bdd(\Ccal)$ then it admits the \newterm{standard t-structure} $(\D^{<0},\D^{\geq 0})$ defined by vanishing of the cochain complex cohomology: $\D^{<0} = \{X \in \Tcal \mid H^i(X) = 0\ \forall i \geq 0\}$ and $\D^{\geq 0} = \{X \in \Tcal \mid H^i(X) = 0\ \forall i < 0\}$. We define the subcategories $\D^{< n}$, $\D^{\geq n}$, $\D^{\leq n}$ and $\D^{> n}$ of $\Tcal$ for $n \in \Zbb$ analogously. Note that in either of the cases $\Tcal = \D(\Mod R)$, $\Tcal = \D^\bdd(\rMod R)$, or $\D^\bdd(\mod R)$ (the last one assumes $R$ right coherent), the standard t-structure $(\D^{\leq 0},\D^{> 0})$ is equal to $(R\Perp{>0},R\Perp{\leq 0})$ and thus induced by the tilting object $R$. Similarly for a choice of an injective cogenerator $W$ of $\Mod R$, the standard t-structure $(\D^{<0},\D^{\geq 0})$ is equal to $(\Perp{\leq 0}W,\Perp{> 0}W)$ and thus induced by the cotilting object $W$. The heart of either of these t-structure is equivalent to $\Mod R$ (resp. to $\mod R$ in the case $\Tcal = \D^\bdd(\mod R)$).

\begin{lemma}[{\cite[Theorem~3.11]{BHM}}]\label{BHM-crit}%
Let $R$ be a ring and $T$ an object of $\D(\rMod R)$. The following are equivalent:

\begin{enumerate}[(a)]
\item $T$ belongs to $\K^\bdd(\Proj R)$ and it is a silting object in $\D(\rMod R)$;

\item $T$ belongs to $\K^\bdd(\Proj R)$, $\Add(T) \subseteq T\Perp{>0}$, and $T$ generates $\D(\rMod R)$.
\end{enumerate}

\noindent Dually, the following are equivalent for $C \in \D(R\lMod)$:

\begin{enumerate}[(a\textsuperscript{co})]
\item $C$ belongs to $\K^\bdd(R\lInj)$ and it is a cosilting object in $\D(R\lMod)$;

\item $C$ belongs to $\K^\bdd(R\lInj)$, $\Prod(C) \subseteq \Perp{>0}C$, and $C$ cogenerates $\D(R\lMod)$.
\end{enumerate}
\end{lemma}

\begin{rmk}\label{silt-DvsDb}
	If $T \in \K^\bdd(\Proj R)$ is a silting object in $\D(\Mod R)$ then it is also a silting object in $\D^\bdd(\Mod R)$, as the induced silting t-structure in $\D(\Mod R)$ restricts to $\D^\bdd(\Mod R)$. \cref{BHM-crit} shows in particular that the converse is also true: An object $T \in \K^\bdd(\Proj R)$ is silting in $\D(\Mod R)$ if and only if it is silting in $\D^\bdd(\Mod R)$. Analogous statement holds for cosilting objects in $\K^\bdd(R\lInj)$, as well as for the tilting and cotilting variants.
\end{rmk}

An object $T \in \D(\rMod R)$ satisfying the condition (a) or (b) above is called a \newterm{silting complex}, dually we have \newterm{cosilting complexes} over $R\lMod$. Silting complexes which are tilting objects are called \newterm{tilting complexes}, similarly we have \newterm{cotilting complexes}. Tilting (resp., cotilting) complexes parametrize bounded derived equivalences to cocomplete abelian categories with a projective generator (resp., to Grothendieck categories) as we now recall. This characterization is for the most part known to experts, see \cite[Theorem A]{PV18} and \cite[Theorem 7.12]{SV18}, which both also apply to larger generality than module categories. However, the first reference does not directly apply to obtain (ii\textsuperscript{co}) below, while we diverge from the latter one by not assuming the derived equivalences of (ii) and (ii\textsuperscript{co}) to extend to unbounded derived categories, a priori. Therefore, we include the following statement and its proof for completeness.

\begin{theorem}\label{tilting-bounded}%
Let $R$ be a ring. Then:

\begin{enumerate}[(i)]
\item Let $T \in \D(\rMod R)$ be a tilting complex. The heart $\Hcal_T$ is a cocomplete abelian category with a projective generator and there is a triangle equivalence $\D^\bdd(\Mod R) \cong \D^\bdd(\Hcal_T)$.

\item Let $\Hcal$ be a cocomplete abelian category with a projective generator and let $\D^\bdd(\Mod R) \cong \D^\bdd(\Hcal)$ be a triangle equivalence. Then there is a tilting complex $T \in \D(\rMod R)$ with $\Hcal_T \cong \Hcal$.

\item[(i\textsuperscript{co})] Let $C \in \D(R\lMod)$ be a cotilting complex. The heart $\Hcal_C$ is a complete abelian category with an injective cogenerator and there is a triangle equivalence $\D^\bdd(R\lMod) \cong \D^\bdd(\Hcal_C)$.

\item[(ii\textsuperscript{co})] Let $\Hcal$ be a complete abelian category with an injective cogenerator and let $\D^\bdd(R\lMod) \cong \D^\bdd(\Hcal)$ be a triangle equivalence. Then there is a cotilting complex $C \in \D(R\lMod)$ with $\Hcal_C \cong \Hcal$.
\end{enumerate}

\noindent In addition, the cotilting heart $\Hcal_C$ from \textup{(i\textsuperscript{co})} or \textup{(ii\textsuperscript{co})} is automatically a Grothendieck category.
\end{theorem}
\begin{proof}
(i) and (i\textsuperscript{co}) are proved in \cite[Corollary~5.2]{PV18}. The proofs of (ii) and (ii\textsuperscript{co}) are analogous, we prove (ii\textsuperscript{co}) here. Let $\alpha\colon\D^\bdd(R\lMod) \toeq \D^\bdd(\Hcal)$ denote the triangle equivalence and let $W$ be an injective cogenerator of $\Hcal$. Set $C = \alpha^{-1}(W) \in \D^\bdd(R\lMod)$. By definition, $W$ is a cotilting object of $\D^\bdd(\Hcal)$ inducing the standard t-structure $(\Perp{\leq 0}W,\Perp{>0}W) = (\D^{<0},\D^{\geq 0})$, thus $C$ is a cotilting object in $\D^\bdd(R\lMod)$. To check that $C$ is a cotilting complex in $\D(R\lMod)$, it suffices in view of \cref{BHM-crit} and \cref{silt-DvsDb} to verify that $C \in \K^\bdd(R\lInj)$. To see this, we claim that the objects of bounded injective dimension in both $\D^\bdd(R\lMod)$ and $\D^\bdd(\Hcal)$ are characterized in terms of the triangulated category structure, and so are preserved and reflected by the equivalence. Indeed, let $\Acal$ be any complete abelian category with enough injectives. Then an object $X \in \D^\bdd(\Acal)$ is of finite injective dimension if and only if for each $Y \in \D^\bdd(\Acal)$ there is $k > 0$ such that $\Hom_{\D^\bdd(\Acal)}(Y,X[i]) = 0$ for all $i \geq k$; this is proved the same way as \cite[Proposition~6.2]{Ri89}. We showed that $C$ is a cotilting complex in $\D(R\lMod)$ and, by the construction, $\Hcal_C \cong \Hcal$.

The final claim follows by \cite[Proposition~3.10]{MV18} and \cite[Theorem~3.6]{AHMV17}.
\end{proof}
\section{Product-complete tilting objects}%
\label{s:prod-comp}%
In this section, let $\Tcal$ be a compactly generated triangulated category with $\Tcal^\cpt$ the full subcategory of compact objects. We start by recalling some notions of the purity theory in $\Tcal$ of Krause \cite{Kr00} and Beligiannis \cite{Bel}. By definition, $\Tcal^\cpt$ is a ringoid (i.e.\ an essentially small preadditive category). Thus we can consider the module category $\Mod{\Tcal^\cpt}$ over the ringoid $\Tcal^\cpt$, that is, the category of all contravariant additive functors $\Tcal^\cpt \to \Mod \Zbb$. We denote by $\yo\colon\Tcal \to \Mod{\Tcal^\cpt}$ the \newterm{restricted Yoneda functor}, which is defined by the assignment $X \mapsto \Hom_\Tcal(-,X)_{\restriction \Tcal^\cpt}$. A morphism $f\colon X \to Y$ in $\Tcal$ is a \newterm{pure monomorphism} if $\yo f$ is a monomorphism in $\Mod{\Tcal^\cpt}$. An object $X \in \Tcal$ is \newterm{pure-injective} if every pure monomorphism $f\colon X \to Y$ in $\Tcal$ is a split monomorphism, or equivalently, if $\yo X$ is an injective object in $\Mod{\Tcal^\cpt}$, and it is \newterm{$\Sigma$-pure-injective} if every object in $\Add(X)$ is pure-injective. Since every object of the form $\yo X$ for $X \in \Tcal$ is \newterm{fp-injective} (see \cref{fp-inj-cogen}) in $\Mod{\Tcal^\cpt}$ by \cite[Lemma 1.6]{Kr00}\footnote{Conversely, every fp-injective object of $\Mod{\Tcal^\cpt}$ is a direct limit of objects in the essential image of $\yo: \Tcal \to \Mod{\Tcal^\cpt}$, see \cite[Lemma 2.7, Theorem 2.8]{Kr00}.}, we also have that $X$ is pure-injective in $\Tcal$ if and only if $\yo X$ is pure-injective in $\Mod{\Tcal^\cpt}$.

Recall that $\Tcal$, being compactly generated, has all coproducts by definition and also all products by \cite[Corollary 1.18]{Nee01}. Extending the classical definition from modules to triangulated setting, we call an object $X \in \Tcal$ \newterm{product-complete} if $\Add(X)$ is closed under taking arbitrary products in $\Tcal$.
\begin{lemma}\label{product-complete}
 If $X \in \Tcal$ is a product-complete object then $X$ is $\Sigma$-pure-injective and $\Add(X) = \Prod(X)$.
\end{lemma}
\begin{proof}
 Since $\yo$ commutes with both products and coproducts, we easily check that $\yo X$ is a product-complete module whenever $X$ is product-complete in $\Tcal$. By a straightforward generalization of \cite[Lemma~2.32]{GT12} to modules over ringoids, $\yo X$ is a $\Sigma$-pure-injective module, and therefore $X$ is $\Sigma$-pure-injective in $\Tcal$.

By the definition of product-completeness, we have $\Prod(X) \subseteq \Add(X)$. The converse inclusion follows by considering the canonical pure monomorphism $X^{(\varkappa)} \to X^\varkappa$ for a cardinal $\varkappa$ and the fact that $X^{(\varkappa)}$ is pure-injective, which follows by the previous paragraph. This implies that $X^{(\varkappa)} \to X^\varkappa$ is a split monomorphism.
\end{proof}

A subcategory $\Ccal$ of $\Tcal$ is called \newterm{definable} if there is a set $\Phi$ of maps in $\Tcal^\cpt$ such that $\Ccal = \{X \in \Tcal \mid \Hom_{\Tcal}(f,X) = 0\ \forall f \in \Phi\}$. The smallest definable subcategory of $\Tcal$ containing an object $X$ will be denoted by $\Def(X)$. Explicitly, $\Def(X) = \{X \in \Tcal \mid \Hom_{\Tcal}(f,X) = 0\ \forall f \in \Phi_X\}$, where $\Phi_X$ consists of all maps $f$ in $\Tcal^\cpt$ such that $\Hom_{\Tcal}(f,X) = 0$.

Assume that $\Tcal$ underlies a compactly generated derivator, see \cite{Lak20} for details. This is a weak assumption allowing computation of homotopy limits and colimits. In case $\Tcal = \D(\rMod R)$ (and, more generally, whenever $\Tcal$ is the homotopy category of a stable model category), homotopy limits and colimits are just the derived functors of limits and colimits. As a consequence, directed colimits being exact in $\Mod R$ (as they are in any Grothendieck category), their derived functor is just the usual directed colimit of complexes (i.e., the directed colimit of modules, applied component-wise), see e.g.\ \cite[Appendix]{HN21} for details. Laking \cite{Lak20} shows that definable subcategories of $\Tcal$ are precisely the subcategories closed under products, pure subobjects and directed homotopy colimits. Equivalently, one can replace directed homotopy colimits by pure quotients in this characterization --- this was proved assuming that $\Tcal$ is algebraic by Laking and Vitória \cite[Theorem 4.7]{LV20}, the algebraic assumption was later removed by Saorín and Šťovíček in \cite[Remark 8.8]{SS20}.

\begin{lemma}\label{product-complete-tilting}%
Assume that $\Tcal$ underlies a compactly generated derivator. Let $T \in \Tcal$ be a product-complete tilting object. Then $\Add(T)$ is a definable subcategory of $\Tcal$.
\end{lemma}
\begin{proof}
By the definition of product-completeness, $\Add(T)$ is closed under products. Consider a pure monomorphism $K \to T^{(\varkappa)}$. Since $T^{(\varkappa)}$ is $\Sigma$-pure-injective by \cref{product-complete}, this map actually splits; this follows from \cite[Corollary~4.4.13]{Prest} and an application of $\yo$. It follows that $\Add(T)$ is closed under pure subobjects and pure quotients, then it is definable by the discussion above.
\end{proof}

\subsection{Topological endomorphism ring of a decent tilting complex}%
We need to briefly recall the recent theory of topological endomorphism rings of tilting complexes from \cite{Hrb22}, which builds upon the work of Positselski and Šťovíček on tilting modules \cite{PS21}. A silting object $T \in \Tcal$ is called \newterm{decent} provided that $\Def(T) \subseteq \Hcal_T$, note that since $\Add(T) \subseteq \Def(T)$, this implies that $T$ is tilting.

\begin{lemma}\label{prodcomp-decent}%
 Assume that $\Tcal$ underlies a compactly generated derivator. If $T$ is a product-complete tilting object of $\Tcal$ then $T$ is decent.
\end{lemma}
\begin{proof}
 By \cref{product-complete-tilting}, $\Add(T)$ is a definable subcategory. Since $T$ is tilting, we have $\Add(T) = \Def(T) \subseteq \Hcal_T$ and so $T$ is decent.
\end{proof}

Let $R$ be a ring and let $T$ be a decent tilting complex in $\D(\rMod R)$. By \cite{Hrb22}, being decent is in this setting equivalent to the character dual complex $C \leqdef T^+ = \RHom_{\Zbb}(T,\Qbb/\Zbb)$ being a cotilting complex in $\D(R\lMod)$. Further following \cite{Hrb22}, the endomorphism ring $\Sfr = \End_{\D(\rMod R)}(T)$ can be endowed with a natural linear topology of open left ideals called the \newterm{compact topology} making it into a complete and separated topological ring. Such a topological ring comes attached with two abelian categories: Positselski's category $\Ctra \Sfr$ of right contramodules \cite{contramodules}---a cocomplete abelian category with a projective generator, and the Grothendieck category $\Sfr\lDisc$ of left discrete modules. The latter category is nothing else then the full subcategory of $\Sfr\lMod$ consisting of modules which are torsion with respect to open left ideals; this is a hereditary pretorsion class inside $\Sfr\lMod$. Let us denote by $\Gamma\colon \Sfr\lMod \to \Sfr\lMod$ the right adjoint to the inclusion $\Sfr\lDisc \subseteq \Sfr\lMod$ postcomposed with the same inclusion, that is, the pretorsion preradical associated to this pretorsion class.

Let us assume in addition that $T$ is \newterm{good}, i.e.\ that $R \in \thick(T)$. Any silting complex is equivalent to a good one and decency is preserved by this \cite[Lemma~4.2]{Hrb22}. Let $A = \dgEnd_R(T)$ be the endomorphism dg~ring of $T$, so that $T$ is an $A$-$R$-dg bimodule. Since $T$ is tilting, $A$ is quasi-isomorphic to $\Sfr$, and there is a triangle equivalence $\epsilon\colon\D(A\ldgMod) \toeq \D(\Sfr\lMod)$ given by the zig-zag of soft truncation morphisms $A \to \tau^{\geq 0}A \gets H^0(A) = \Sfr$. Here, $\Sfr\lMod$ is identified with the heart of the standard t-structure in $\D(A\ldgMod)$. By \cite[Theorem~5.4]{Hrb22}, there is an equivalence $\psi: \D^\bdd(R\lMod) \cong \D^\bdd(\Sfr\lDisc): \psi^{-1}$, where $\psi$ is induced by corestriction of the functor $\epsilon \circ (T \otimes_R^\mathbf{L} -)$ and $\psi^{-1}$ is the restriction of $\RHom_A(T,-) \circ \epsilon^{-1}$. This equivalence further restricts to an equivalence of abelian categories
\[
	H^0(T \otimes_R^\mathbf{L} -):\Hcal_C\buildrel\cong\over\longrightarrow
		\Sfr\lDisc: \RHom_A(T,-).
\]
Similarly by \cite[Theorem~4.7]{Hrb22} we have an equivalence $\Hcal_T \cong \Ctra \Sfr$ which restricts to an equivalence $\Add(T) \cong \Ctrapr \Sfr$, the latter being the full subcategory of $\Ctra \Sfr$ consisting of all projective right $\Sfr$-contramodules.

We start by adding to the general results of \cite{Hrb22} that the linear topology on $\Sfr$ is in this situation actually a Gabriel topology of finite type, or equivalently, $\Sfr\lDisc$ is closed under extensions in $\Sfr\lMod$ (so it is a hereditary torsion class) whose torsion radical $\Gamma$ commutes with direct limits.

\begin{prop}\label{gamma-limits}%
 In the setting above, the pretorsion preradical $\Gamma\colon \Sfr\lMod \to \Sfr\lMod$ is a torsion radical and it commutes with direct limits. In particular, the left open ideals of $\Sfr$ form a Gabriel topology with a base of finitely generated left ideals of $\Sfr$. 
\end{prop}
\begin{proof}
 First, by \cite[Theorem~5.4]{Hrb22}, $T$ good implies that the forgetful functor $\D^\bdd(\Sfr\lDisc) \to \D^\bdd(\Sfr\lMod)$ is fully faithful. As a consequence, for any $M,N \in \Sfr\lDisc$ we have a natural isomorphism $\Ext^1_{\Sfr\lDisc}(M,N) \cong \Ext^1_{\Sfr}(M,N)$. It follows that $\Sfr\lDisc$ is an extension closed subcategory of $\Sfr\lMod$, thus it forms a torsion class, and so $\Gamma$ is a torsion radical.

Following the proof of \cite[Theorem~5.4]{Hrb22}, there is a commutative square of triangle functors
\[
\xymatrix{%
	\D^\bdd(\Hcal_C) \ar[d]_-{\real_C}^-{\cong}\ar[r]_-{\cong}^-{\D^\bdd(F)} &
		\D^\bdd(\Sfr\lDisc)\ar[d]^-{\iota} \cr
	\D^\bdd(R\lMod) \ar[r]^-{\psi} & \D^\bdd(\Sfr\lMod) 
}
\]
where $F = H^0(T \otimes_R^\mathbf{L}-)$ is the exact equivalence $\Hcal_C \toeq \Sfr\lDisc$ and $\D^\bdd(F)$ its extension to bounded derived categories, $\psi = \epsilon \circ (T \otimes_R^\mathbf{L} -)$, $\real_C$ is the realization functor with respect to the t-structure $(\Perp{\leq 0}C,\Perp{>0}C)$ and a suitable f-enhancement, and $\iota\colon\D^\bdd(\Sfr\lDisc) \to \D^\bdd(\Sfr\lMod)$ is the forgetful functor, which we know to be fully faithful. By taking right adjoints, we see that $\iota$ has a right adjoint naturally equivalent to the functor $\rho = \D^\bdd(F) \circ {\real_C^{-1}} \circ \phi$, where $\phi = \RHom_A(T,-) \circ \epsilon^{-1}$ is the right adjoint to $\psi$. 
  
We claim that $\Gamma \cong \iota \circ H^0(\rho)$. Let $X \in \D^\bdd(\Sfr\lDisc)$ and $M \in \Sfr\lMod$. By the adjunction, we have an isomorphism $\Hom_{\D^\bdd(\Sfr)}(X,M) \cong \Hom_{\D^\bdd(\Sfr\lDisc)}(X,\rho M)$. It follows that $\Hom_{\D^\bdd(\Sfr\lDisc)}(N[i],\rho M) = 0$ for any $N \in \Sfr\lDisc$ and $i>0$, which implies that $H^i(\rho M) = 0$ for all $i<0$. Then for any $N \in \Sfr\lDisc$ we have 
\begin{align*}
	\Hom_{\Sfr}(N,M) &\cong \Hom_{\D^\bdd(\Sfr\lDisc)}(N,\rho M) \cr
	&\cong \Hom_{\Sfr\lDisc}(N,\tau^{\leq 0}(\rho M)) = \Hom_{\Sfr\lDisc}(N,H^0(\rho M)).
\end{align*}
In other words, $H^0(\rho)\colon\Sfr\lMod \to \Sfr\lDisc$ is the right adjoint to the inclusion $\Sfr\lDisc \subseteq \Sfr\lMod$, and so $\iota \circ H^0(\rho)$ is equivalent to $\Gamma$.

Finally, we have for any $M \in \Sfr\lMod$ that 
\begin{align*}
	H^0(\rho(M))=H^0(\D^\bdd(F)(\real_C^{-1}\phi(M))) &\cong F(H^0_C(\phi(M))) \cr
	&= F(H^0_C(\RHom_A(T,M))) 
\end{align*}
(here we use \cite[Theorem~3.11(i)]{PV18}). The functor 
\[
	F(H^0_C(\RHom_A(T,-)))\colon \Sfr\lMod\longrightarrow \Sfr\lDisc
\] 
clearly commutes with direct limits in $\Sfr\lMod$, because $T$ is a compact object in $\D(A\ldgMod)$ and direct limits in $\Sfr\lMod$ coincide with directed homotopy colimits computed inside $\D(A\ldgMod)$, while both the exact equivalence $F$ and the cohomological functor $H^0_C$ are known to preserve directed (homotopy) colimits. Then also $\Gamma$ commutes with direct limits, as the inclusion $\Sfr\lDisc \subseteq \Sfr\lMod$ clearly preserves direct limits.

By \cite[\S VI.5]{SS75}, the left open ideals form a Gabriel topology, and then by \cref{gamma-limits} and \cite[\S XIII, Proposition~1.2]{SS75}, the topology has a base of finitely generated left ideals of $\Sfr$.
\end{proof}

\subsection{Locally coherent and locally coperfect abelian categories}\label{ss:loccoh}%
Recall that an essentially small abelian category $\Acal$ is \newterm{noetherian} (resp.\ \newterm{artinian}) if every object in $\Acal$ is noetherian (resp., artinian), which means that it satisfies a.c.c.\ (resp., d.c.c.) on its subobjects. Let $\Ccal$ be a locally finitely presentable abelian category and let $\fp \Ccal$ denote the (essentially small) full subcategory of finitely presentable objects of $\Ccal$. We recall that this automatically renders $\Ccal$ a Grothendieck category. We call $\Ccal$ \newterm{locally noetherian} if $\Ccal$ admits a generating set of noetherian objects. It can be easily seen that $\Ccal$ is locally noetherian if and only if $\fp \Ccal$ is a noetherian abelian category. In particular, $\Ccal$ is \newterm{locally coherent}, which by definition means that $\fp \Ccal$ is itself an abelian category. A natural question is what properties of $\Ccal$ characterize the case in which $\fp{\Ccal}$ is an artinian abelian category. It turns out that this occurs precisely when $\Ccal$ is locally coherent and \newterm{locally coperfect}. The latter property means that there is a set of generators in $\Ccal$ which are \newterm{coperfect}, which means that they satisfy d.c.c.\ on \textit{finitely generated} subobjects. This is in fact equivalent to every object of $\Ccal$ being coperfect. For details, we refer the reader to \cite{Roos69} and \cite{PS19b}.

We have the following result of Roos \cite{Roos69}, which can be seen as a large category version of the obvious explicit duality $\Acal \mapsto \Acal^\op$ between noetherian and artinian abelian categories. Given an essentially small abelian category $\Acal$, we let $\Lex(\Acal)$ be the abelian category of all left exact additive functors $\Acal \to \Mod \Zbb$; this is a locally coherent abelian category which satisfies $\fp{\Lex(\Acal)} \cong \Acal$, see \cite{Roos69} and also \cite[Proposition~13.2]{PS19b}.

\begin{theorem}[{\cite{Roos69}, \cite{PS19b}}]\label{Roos}
There is a bijective correspondence
\[
	\left\lbrace\begin{array}{@{}c@{}}%
		\text{Locally noetherian} \\
		\text{abelian categories} \\
		\text{up to equivalence}
	\end{array}\right\rbrace
	\buildrel1\mathchar`-1\over\longleftrightarrow 
	\left\lbrace\begin{array}{@{}c@{}}
		\text{Locally coherent and} \\
		\text{locally coperfect} \\
		\text{abelian categories} \\
	\text{up to equivalence}
	\end{array}\right\rbrace
\]
 which is defined in both directions by the assignment $\Ccal \mapsto \Lex(\fp{\Ccal}^{\op})$.
 
 Furthermore, any locally coherent, locally coperfect abelian category $\Ccal$ is equivalent to $\Sfr\lDisc$ for a suitable complete and separated left topological ring $\Sfr$.
\end{theorem}

It turns out that the product-completeness of $T$ characterizes the situation in which the induced cotilting heart $\Hcal_C$ is locally coherent and locally coperfect. We recall that the categories $\Ctra \Sfr$ of right $\Sfr$-contramodules and $\Sfr\lDisc$ of left discrete $\Sfr$-modules are paired by a bifunctor $- \ctrtensor_\Sfr -:\Ctra \Sfr \times \Sfr\lDisc \to \Mod \Zbb$ called the \newterm{contratensor product}, which shares many properties with the usual tensor functor in case of ordinary module categories, see \cite[\S 1]{PS19b} for the basic reference. In the case when the forgetful functor $\Ctra \Sfr \to \Mod \Sfr$ is fully faithful, the contratensor product $- \ctrtensor_\Sfr -$ actually coincides with the restriction of the ordinary tensor product $- \otimes_\Sfr -$, this happens in our setting whenever the tilting complex $T$ is good, see \cite[Lemma~7.11]{PS21} and \cite[Lemma 5.1]{Hrb22}. A right $\Sfr$-contramodule $\Mfr$ is \newterm{flat} if the functor $\Mfr \ctrtensor_\Sfr -: \Sfr\lDisc \to \Mod \Zbb$ is exact, any projective right $\Sfr$-contramodule is flat \cite[\S 14]{PS19b}.

\begin{theorem}\label{loc-coh-cop}%
Let $R$ be a ring, $T$ a tilting complex in $\D(\Mod R)$, and $C = T^+$ its dual cosilting complex in $\D(R\lMod)$. Then:

\begin{enumerate}[(i)]
\item If $T$ is product-complete then $T$ is decent and $\Hcal_C$ is a locally coherent and locally coperfect abelian category.

\item If $T$ is decent and $\Hcal_C$ is locally coherent and locally coperfect then $T$ is product-complete.
\end{enumerate}
\end{theorem}
\begin{proof}
Since product-completeness and decency of $T$, as well as the equivalence class of $\Hcal_C \cong \Sfr\lDisc$, are invariant the under change of the tilting complex $T$ up to equivalence, we can without loss of generality assume that $T$ is good \cite[Lemma~4.2, \S 6.2]{Hrb22}.

\subparagraph{(i)} Recall first from \cref{product-complete-tilting} that $T$ is decent so that $C$ is cotilting.

Let us first show that $\Hcal_C \cong \Sfr\lDisc$ is locally coherent. By \cref{gamma-limits}, it is locally finitely presentable. Also by \cref{gamma-limits}, $\Gamma$ commutes with direct limits (as observed above), and so the finitely presentable objects in $\Sfr\lDisc$ are precisely the objects which are finitely presented as left $\Sfr$-modules. Let $0 \to K \to M \to N$ be an exact sequence with $M,N$ finitely presentable objects in $\Sfr\lDisc$. For any cardinal $\varkappa$, consider the commutative diagram
\[
\xymatrix{%
  0 \ar[r] &  \Sfr^\varkappa \otimes_\Sfr K \ar[r]\ar[d] &
  	\Sfr^\varkappa \otimes_\Sfr M \ar[r]\ar[d] &
  	\Sfr^\varkappa \otimes_\Sfr N \ar[d] \cr
  0 \ar[r] & K^\varkappa \ar[r] & M^\varkappa \ar[r] & N^\varkappa
}
\]
Here, the vertical arrows are the natural ones; the rows of the diagram are exact, because the tensor product $- \otimes_\Sfr -$ here coincides with the contratensor product $- \ctrtensor_\Sfr -$ and because $\Sfr^\varkappa$ is a flat right $\Sfr$-contramodule. The latter fact follows from $\Prod(\Sfr) \subseteq \Add(\Sfr)$ in $\Ctra \Sfr$ which reflects the assumption $\Prod(T) \subseteq \Add(T)$ in $\D(\rMod R)$ (note that both the $\Add$- and the $\Prod$-closure of the projective object $T$ is computed the same in $\D(\rMod R)$ and in the abelian category $\Hcal_T \cong \Ctra \Sfr$). Since $M,N$ are finitely presented left $\Sfr$-modules, the two rightmost vertical maps are isomorphisms, and therefore so is the leftmost vertical map. It follows by a standard argument (\cite[Proposition~10.89.3.]{Stacks}) that $K$ is a finitely presented left $\Sfr$-module, and therefore it is a finitely presentable object of $\Sfr\lDisc$. We proved that $\Sfr\lDisc$ is locally coherent.

By \cref{product-complete-tilting}, $\Add(T)$ is closed under directed homotopy colimits. It follows from \cite[Lemma~7.3]{SSV17} that directed homotopy colimits in $\Add(T)$ coincide with direct limits of objects in $\Add(T)$ computed in the heart $\Hcal_T$. In view of the equivalence $\Hcal_T \cong \Ctra \Sfr$ which restricts to $\Add(T) \cong \Ctrapr{\Sfr}$, we see that the category of projective right $\Sfr$-contramodules is closed under direct limits computed in $\Ctra \Sfr$. By \cite[Theorem~14.1]{PS19b}, the topological ring $\Sfr$ is therefore topologically right perfect and $\Sfr\lDisc$ is locally coperfect by \cite[Theorem~14.4]{PS19b}.

\subparagraph{(ii)} By definition \cite{PS19b}, $\Sfr$ is topologically left coherent. Using \cite[Theorem~14.12]{PS19b} and \cite[Theorem~14.1]{PS19b}, we see that $\Sfr$ is also topologically right perfect, which in particular means that $\Ctrapr{\Sfr}$ coincides with the class of flat right $\Sfr$-contramodules. It is enough to see that this class is closed under products. Let $(\mathfrak{F}_i \mid i \in I)$ be a collection of flat right $\Sfr$-contramodules. If $N \in \Sfr\lDisc$ is a coherent left discrete $\Sfr$-module then we have the following isomorphisms:
\[
	(\prod_{i \in I}\mathfrak{F}_i) \ctrtensor_\Sfr N \cong
		(\prod_{i \in I}\mathfrak{F}_i) \otimes_\Sfr N \cong
		\prod_{i \in I}(\mathfrak{F}_i \otimes_\Sfr N) \cong
		\prod_{i \in I}(\mathfrak{F}_i \ctrtensor_\Sfr N);
\]
the first and last one follow again by $T$ begin good and \cite[Lemma~5.1]{Hrb22}, while the middle one follows since $N$ is a finitely presented left $\Sfr$-module. A standard argument using local coherence of $\Sfr\lDisc$ then shows that $\prod_{i \in I}\mathfrak{F}_i$ is a flat right $\Sfr$-contramodule. Indeed, by \cite[Lemma~5.9]{Kr98b}, any short exact sequence in $\Sfr\lDisc$ can be written as a direct limit of short exact sequences in $\fp{\Sfr\lDisc}$, and so the functor $\prod_{i \in I}\mathfrak{F}_i \ctrtensor_\Sfr -: \Sfr\lDisc \to \Mod \Zbb$ is exact.
\end{proof}

\section{Commutative noetherian rings and codimension filtrations}%
\label{s:comnoeth}%
From now on, $R$ is a commutative noetherian ring with Zariski spectrum $\Spec R$. We also abbreviate $\D(R)\leqdef \D(\rMod R)$, $\D^\bdd(R)\leqdef \D^\bdd(\rMod R)$, and $\D^\bdd_\fg(R)\leqdef\D^\bdd(\mod R)$; note that the last category is well-known to be equivalent to the full subcategory of $\D^\bdd(R)$ consisting of complexes with finitely generated cohomology.

\subsection{Compactly generated t-structures in $\D(R)$}%
\label{ss:ajs}%
Alonso, Jerem\'\i as and Saor\'\i n \cite{ATJLS10} established a bijection between compactly generated t-structures in $\D(R)$ and sp-filtrations of the Zariski spectrum $\Spec R$. An \newterm{sp-filtration} is a function $\Phi$ assigning to every integer $n \in \Zbb$ a specialization-closed subset $\Phi(n)$ of $\Spec R$, that is, an upper subset of the poset $(\Spec R, \subseteq)$, such that $\Phi(n-1) \supseteq \Phi(n)$ for each $n \in \Zbb$. If $\Phi$ is an sp-filtration, then the corresponding t-structure $(\Ucal,\Vcal)$ is determined by 
\begin{align*}
	\Ucal &= \{X \in \D(R) \mid \Supp(H^n(X)) \subseteq \Phi(n)\ \forall n \in \Zbb\}, \cr
	\intertext{where $\Supp(M) = \{\pp \in \Spec R \mid M_\pp \neq 0\}$, and}
	\Vcal &= \{X \in \D(R) \mid \RGamma_{\!\Phi(n)}(X) \in \D^{> n}\ \forall n \in \Zbb\},
\end{align*}
where $\RGamma_{\!\Phi(n)}\colon\D(R) \to \D(R)$ is the right derived functor of the torsion functor $\Gamma_{\!\Phi(n)}\colon\Mod R \to \Mod R$ with respect to the support $\Phi(n)$. Note that the coaisle $\Vcal$ can also be described using depth (see e.g. \cite[\S 2.3]{HNS}):
\[
	\Vcal = \{X \in \D(R) \mid
		\depth_{R_\qq}X_\qq > n\ \forall \qq \in \Phi(n)\ \forall n \in \Zbb\}.
\]

A t-structure $(\Ucal,\Vcal)$ is \newterm{non-degenerate} if $\bigcap_{n \in \Zbb}\Ucal[n] = 0 = \bigcap_{n \in \Zbb}\Vcal[n]$ and it is \newterm{intermediate} if $\D^{< n} \subseteq \Ucal \subseteq \D^{< m}$ for some integers $n \leq m$; any intermediate t-structure is non-degenerate. Both these properties of a compactly generated t-structure in $\D(R)$ are easily read from the corresponding sp-filtration $\Phi$: we call $\Phi$ non-degenerate if $\bigcup_{n \in \Zbb}\Phi(n) = \Spec R$ and $\bigcap_{n \in \Zbb}\Phi(n)=\emptyset$, while we call $\Phi$ intermediate if $\Phi(n) = \Spec R$ and $\Phi(m)=\emptyset$ for some integers $n<m$, see \cite[Theorem~3.8]{AHH19}.

The t-structures in $\D(R)$ induced by pure-injective cosilting objects are precisely the non-degenerate compactly generated t-structures \cite{HN21}. Moreover, t-structures in $\D(R)$ induced by cosilting complexes are precisely the intermediate compactly generated t-structures. Furthermore, the assignment $T \mapsto T^+$ yields a bijection between the equivalence classes of silting and cosilting complexes. See \cite[Theorem~3.8]{AHH19}.

\subsection{Restrictable t-structures}%
Recall that a t-structure $(\Ucal,\Vcal)$ in $\D(R)$ is \newterm{restrictable} if the restricted pair $(\Ucal \cap \D^\bdd_\fg(R),\Vcal \cap \D^\bdd_\fg(R))$ yields a t-structure in the triangulated category $\D^\bdd_\fg(R)$. A non-degenerate sp-filtration $\Phi$ satisfies the \newterm{weak Cousin condition} if whenever $\pp \subsetneq \qq$ is a minimal inclusion of primes and $\qq \in \Phi(n)$ then $\pp \in \Phi(n-1)$. In the following, we gather several important results about restrictable t-structures in $\D(R)$.

\begin{theorem}\label{restrictable-props}%
Let $R$ be a commutative noetherian ring and $(\Ucal,\Vcal)$ a compactly generated t-structure in $\D(R)$ with heart $\Hcal$ corresponding to an sp-filtration $\Phi$. Then:

\begin{enumerate}[(i)]
\item if $(\Ucal,\Vcal)$ is intermediate then $(\Ucal,\Vcal)$ is restrictable if and only if $(\Ucal,\Vcal)$ is cotilting and $\Hcal$ is locally coherent;

\item if $(\Ucal,\Vcal)$ is restrictable then $\Phi$ satisfies the weak Cousin condition;

\item if $R$ is CM-excellent (see \cref{s:cmexc}) then the converse implication of (ii) holds as well;

\item The restriction assignment induces a bijection
\[
	\left\lbrace\begin{array}{@{}c@{}}
		\text{Restrictable compactly generated} \\
		\text{t-structures in $\D(R)$}
	\end{array}\right\rbrace
	\buildrel1\mathchar`-1\over\longleftrightarrow
	\left\lbrace\begin{array}{@{}c@{}}
		\text{t-structures in} \\
		\text{$\D^\bdd_\fg(R)$}
	\end{array}\right\rbrace.
\]
\end{enumerate}
\end{theorem}
\begin{proof}\leavevmode
\subparagraph{(i)} The direct implication is \cite[Corollary~6.17]{PV20} and \cite[Theorem~6.3]{Sao17}, while the converse is proved in \cite[Theorem~3.13]{HP}.

\subparagraph{(ii)} This is \cite[Corollary~4.5]{ATJLS10}.

\subparagraph{(iii)} This has recently been proved by Takahashi \cite{Tak22}, the special case when $R$ has a classical dualizing complex was proved in \cite[\S 6]{ATJLS10}.

\subparagraph{(iv)} See \cite[Corollary~3.12]{ATJLS10} or more generally \cite[Corollary~4.2]{MZ23}.
\end{proof}

\subsection{Module-finite algebra extensions}%
Let $\lambda\colon R \to A$ be a noetherian commutative $R$-algebra and let $\Spec{\lambda}\colon\Spec A \to \Spec R$ denote the induced map on spectra. Given an sp-filtration $\Phi$ on $\Spec R$ we can define an induced sp-filtration $\lambda\Phi$ on $\Spec A$ by setting $\lambda\Phi(n) = \Spec{\lambda}^{-1} (\Phi(n))$ for all $n \in \Zbb$. This way, $\lambda$ transfers compactly generated t-structures in $\D(R)$ to compactly generated t-structures in $\D(A)$, see \cite[\S 5]{BHM} for details. More explicitly, let $(\Ucal,\Vcal)$ be the compactly generated t-structure in $\D(R)$ corresponding to $\Phi$ and let $(\lambda^* \Ucal,\lambda^! \Vcal)$ denote the compactly generated t-structure in $\D(A)$ corresponding to $\lambda\Phi$ (the notation is justified\footnote{One can in fact show that the t-structure $(\lambda^* \Ucal,\lambda^! \Vcal)$ is generated by the image of (compact objects in) $\Ucal$ under $\lambda^*$, this follows essentially from \cite[\S 4, \S 5]{BHM}.} in \cref{algebra-restr} below). Then 
\begin{align*}
	\Ucal &= \{X \in \D(R) \mid
		\Supp H^n(X) \subseteq \Phi(n)\ \forall n \in \Zbb\} \cr
	\noalign{\hbox{and}}
	\lambda^* \Ucal &= \{X \in \D(A) \mid
		\Supp H^n(X) \subseteq \lambda\Phi(n)\ \forall n \in \Zbb\}.
\end{align*}
Let 
$$\lambda_*  \colon \D(A) \to \D(R)$$ 
denote the forgetful functor, which admits the left adjoint 
$$\lambda^* = (- \otimes_R^\mathbf{L} A) \colon \D(R) \to \D(A)$$ 
and the right adjoint $$\lambda^! = \RHom_R(A,-) \colon \D(R) \to \D(A).$$ 
An $R$-algebra $\lambda\colon R \to A$ is \newterm{module-finite} if $A$ is finitely generated as an $R$-module. Note that if $\lambda$ is module-finite and $Y \in \D(A)$ then $Y \in \D^\bdd_\fg(A)$ if and only if $\lambda_* Y \in \D^\bdd_\fg(R)$.

\begin{lemma}\label{algebra-restr}%
Let $\lambda\colon R \to A$ be a commutative module-finite $R$-algebra. In the notation set above, the following hold for any $X \in \D(R)$ and $Y \in \D(A)$:

\begin{enumerate}[(i)]
 \item $Y \in \lambda^* \Ucal$ if and only if $\lambda_* Y \in \Ucal$.

 \item If $X \in \Ucal$ then $\lambda^* X \in \lambda^* \Ucal$.
 
 \item If $X \in \Vcal$ then $\lambda^! X \in \lambda^! \Vcal$.
\end{enumerate} 
\end{lemma}
\begin{proof}
(i).\ Let $M = H^n(Y)$ for some $n \in \Zbb$. Then condition (i) just says that $\Supp(M) \subseteq \lambda\Phi(n)$ is equivalent to $\Supp(\lambda_* M) \subseteq \Phi(n)$. By writing $M$ as a directed union of finitely generated submodules, and using that $\lambda$ is module-finite, we can clearly assume that $M$ is finitely generated. Then the statement follows from \cite[Lemma~10.40.6.]{Stacks}.

\subparagraph{(ii)} By (i), it suffices to show that $\lambda_* \lambda^* X \in \Ucal$. This follows from \cite[Proposition~2.3(i)]{Hrb20}.

\subparagraph{(iii)} For any $Y \in \lambda^* \Ucal$ we have $\Hom_{\D(A)}(Y,\lambda^! X) \cong \Hom_{\D(R)}(\lambda_*Y,X) = 0$ by (i), which shows that $\lambda^! X \in \lambda^* \Ucal \Perp{0} = \lambda^!\Vcal$.
\end{proof}

\begin{prop}\label{restrictable}%
Let $\lambda\colon R \to A$ be a commutative module-finite $R$-algebra. Let $(\Ucal,\Vcal)$ be an intermediate restrictable t-structure in $\D(R)$ and $(\lambda^* \Ucal,\lambda^! \Vcal)$ the induced t-structure in $\D(A)$. Then $(\lambda^* \Ucal,\lambda^! \Vcal)$ is restrictable in $\D(A)$.
\end{prop}
\begin{proof}
 Let $X \in \D^\bdd_\fg(A)$ and consider the approximation triangle 
\[
	\wbar{U}\buildrel f\over\longrightarrow X\longrightarrow
		\wbar{V}\buildrel+\over\longrightarrow
\]
 with respect to $(\lambda^* \Ucal,\lambda^! \Vcal)$ in $\D(A)$, as well as the approximation triangle 
\[
	U\buildrel h\over\longrightarrow \lambda_* X \longrightarrow
		V\buildrel+\over\longrightarrow
\]
with respect to $(\Ucal,\Vcal)$ in $\D(R)$. Since $\lambda\colon R \to A$ is module-finite, $\lambda_* X \in \D^\bdd_\fg(R)$. By the assumption, we know that $U \in \D^\bdd_\fg(R)$, and the goal is to prove that $\wbar{U} \in \D^\bdd_\fg(A)$. Since the t-structure $(\Ucal,\Vcal)$ is intermediate, it is easy to see that so is $(\lambda^* \Ucal,\lambda^! \Vcal)$. Then clearly $\wbar{U} \in \D^\bdd(A)$.
 
For any object $M \in \D(R)$ the natural map
\[
	\eta_M = M \otimes_R^\mathbf{L} \lambda\colon M\longrightarrow
		\lambda_* \lambda^* M
\]
in $\D(R)$ is the unit of the adjunction. It follows that for any $\wbar{M} \in \D(A)$ and any map $s\colon M \to \lambda_* \wbar{M}$ in $\D(R)$, there is a map $\wtilde{s}\colon \lambda^* M  \to \wbar{M}$ in $\D(A)$ such that $s = \lambda_*(\wtilde{s}) \eta_M$.

Consider the map $\wtilde{h}\colon \lambda^* U  \to X$ induced from $h$ as above. Since $\lambda^* U  \in \lambda^* \Ucal$ by \cref{algebra-restr}, there is a (essentially unique) map $l\colon \lambda^* U \to \wbar{U}$ in $\D(A)$ such that $\wtilde{h}=fl$. Similarly, there is a (essentially unique) map $g\colon \lambda_*\wbar{U} \to U$ in $\D(R)$ such that $\lambda_*(f) = hg$. Consider the composition $\eta_U g\colon \lambda_*\wbar{U} \to \lambda_* \lambda^* U$, then there is again a map $\gamma = \widetilde{\eta_U g}\colon \lambda^*\lambda_*\wbar{U}  \to \lambda^* U $ in $\D(A)$ such that $\eta_U g = \lambda_*(\gamma) \eta_{\lambda_*\wbar{U}}$. Consider the composition:
\[
	\lambda_*\wbar{U}\buildrel\eta_{\lambda_*\wbar{U}}\over\longrightarrow
		\lambda_*\lambda^*\lambda_*\wbar{U}\buildrel\lambda_*(\gamma)\over\longrightarrow
		\lambda_* \lambda^* U \buildrel \lambda_*(l) \over\longrightarrow \lambda_*\wbar{U}.
\]
First, because $\eta_{\lambda_*\wbar{U}} = \lambda_*\wbar{U} \otimes_R^\mathbf{L} \lambda$, the map $\eta_{\lambda_*\wbar{U}}$ is in the essential image of the forgetful functor $\lambda_*$, and therefore the composition $\lambda_*(l) \lambda_*(\gamma) \eta_{\lambda_*\wbar{U}}$ is of the form $\lambda_*(e)$ for some $e \in \End_{\D(A)}(\wbar{U})$. We compute: $\lambda_*(f) \lambda_*(l) \lambda_*(\gamma) \eta_{\lambda_*\wbar{U}} = \lambda_*(\wtilde{h}) \eta_U g = hg = \lambda_*(f)$. It follows that $fe = f$, and therefore $e$ is an automorphism of $\wbar{U}$. As a consequence, $l\colon \lambda^* U  \to \wbar{U}$ is a split epimorphism in $\D(A)$. Since $U \in \D^\bdd_\fg(R)$ and $A$ is a finitely generated $R$-module, it follows that every cohomology of $\lambda^* U $ is finitely generated over $R$, and therefore also over $A$. Then the same is true for $\wbar{U}$. Since we already know that $\wbar{U}$ is cohomologically bounded, the proof is concluded.
\end{proof}

The following characterization of when the cotilting property passes to factor rings of $R$ is to some extent implicit in \cite[\S 7]{HNS}.

\begin{lemma}\label{cotilting-flat}%
Let $C$ be a cosilting complex in $\D(R)$. The following are equivalent:
 
\begin{enumerate}[(a)]
\item $C$ is cotilting and $\Hom_{\D(R)}(C^\varkappa,C)$ is flat as an $R$-module for any cardinal $\varkappa$;

\item for each ideal $I$ of $R$, $\RHom_R(R/I,C)$ is a cotilting object in $\D(R/I)$. 
\end{enumerate}
\end{lemma}
\begin{proof}
By a general argument \cite[Theorem~4.2(II)(1)]{BHM}, $\wbar{C} \leqdef \RHom_R(R/I,C)$ is a cosilting object in $\D(R/I)$. Arguing similarly as in the proof of \cite[Proposition~7.4]{HNS}, \cite[Proposition 2.1(ii)]{CH09} yields an isomorphism
\[
	\RHom_R(C^\varkappa,C) \otimes_R^\mathbf{L} R/I \cong
		\RHom_{R/I}(\wbar{C}^\varkappa,\wbar{C})
\]
for any cardinal $\varkappa$. Since $C$ is cotilting, we have $\RHom_R(C^\varkappa,C) \cong\Hom_{\D(R)}(C^\varkappa,C)$ in $\D(R)$. If $\Hom_{\D(R)}(C^\varkappa,C)$ is flat, then the cohomology of $\RHom_{R/I}(\wbar{C}^\varkappa,\wbar{C})$ is concentrated in degree zero, and thus the cosilting complex $\wbar{C}$ is cotilting in $\D(R/I)$. Conversely, assume that $\RHom_R(R/I,C)$ is cotilting in $\D(R/I)$ for all ideals $I$. Then by the isomorphism above, $\Tor_i^R(R/I,\Hom_{\D(R)}(C^\varkappa,C)) = 0$ for all ideals $I$ and $i>0$, and thus $\Hom_{\D(R)}(C^\varkappa,C)$ is a flat $R$-module by the flat test.
\end{proof}

\begin{rmk}
In \cite[\S 7]{HNS}, also the dual condition of $\Hom_{\D(R)}(T,T^{(\varkappa)})$ being flat for all cardinals $\varkappa$ is considered for a tilting complex $T$. Analogously to \cref{cotilting-flat}, one can show that this condition is equivalent to $T \otimes_R^\mathbf{L} R/I$ being a tilting complex in $\D(R/I)$ for all ideals $I$. We remark that in light of \cite[Corollary~3.7]{Hrb22}, there is now a one-way relation between these two conditions for a pair of a tilting complex $T$ and its dual cotilting complex $T^+$ in $\D(R)$. Indeed, 
\[
	\Hom_{\D(R)}((T^+)^\varkappa,T^+) \hbox{ is flat for all $\varkappa$}\quad\Longrightarrow\quad
		\Hom_{\D(R)}(T,T^{(\varkappa)}) \hbox{ is flat for all $\varkappa$}.
\]
Whether the converse is true in this setting remains unclear.
\end{rmk}

In what follows, we show that the equivalent conditions of \cref{cotilting-flat} are strongly connected to the restrictability of the induced cotilting t-structure.

\begin{lemma}\label{pi-pure-split}
	A triangle $X \to Y \to Z \buildrel+\over\to$ is pure in $\D(R)$ if and only if the induced triangle $\RHom_R(Z,I) \to \RHom_R(Y,I) \to \RHom_R(X,I) \buildrel+\over\to$ is split for any pure-injective object $I \in \D(R)$.
\end{lemma}
\begin{proof}
	The ``if'' statement follows from \cite[Lemma 2.6(iii)]{AHH19} by setting $I = R^+$. Let us prove the ``only if'' statement. If $I = X^+$ for some $X \in \D(R)$ then $\RHom_R(-,X^+) \cong (- \otimes_R^\mathbf{L} X)^+$, and thus the triangle is split by an application of \cite[Lemma 2.6(ii),(iii)]{AHH19}. Finally, the case of a general pure-injective object $I$ reduces to the previous one because there is a natural morphism $I \to I^{++}$ which is a split monomorphism by \cite[Lemma 2.7]{AHH19}.
\end{proof}

\begin{lemma}\label{cotilting-coloc}%
Let $C$ be a cotilting complex in $\D(R)$. Then for any $\pp \in \Spec R$, $\RHom_R(R_\pp,C)$ is a cotilting complex in $\D(R_\pp)$.
\end{lemma}
\begin{proof}
This follows similarly as \cite[Lemma 5.10]{HHZ21}, but we also provide a direct proof. As above, \cite[Theorem~4.2(II)(1)]{BHM} yields that $\RHom_R(R_\pp,C)$ is a cosilting complex in $\D(R_\pp)$, so it suffices to show that $\RHom_R(R_\pp,C)^\varkappa \in \Perp{<0}\!\RHom_R(R_\pp,C)$ for any cardinal $\varkappa$. By adjunction, this reduces to showing that $\RHom_R(R_\pp,C^\varkappa) \in \Perp{<0}C$. We show more generally that $\RHom_R(F,C^\varkappa) \in \Perp{<0}C$, where $F$ is a flat $R$-module. Let $\pi\colon R^{(\lambda)} \to F$ be an epimorphism for some cardinal $\lambda$. Since $F$ is flat, $\pi$ is a pure epimorphism. Since $C^\varkappa$ is pure-injective, the induced morphism $\RHom_R(\pi,C^\varkappa)\colon\RHom_R(F,C^\varkappa) \to \RHom_R(R^{(\lambda)},C^\varkappa)$ is a split monomorphism in $\D(R)$ by \cref{pi-pure-split}. As $C$ is cotilting, we have that $\RHom_R(R^{(\lambda)},C^\varkappa) \cong (C^\varkappa)^\lambda$ belongs to $\Perp{<0}C$, and thus $\RHom_R(F,C^\varkappa) \in \Perp{<0}C$ as desired.
\end{proof}

\begin{prop}\label{restrictable-flat}%
Let $C$ be a cotilting complex in $\D(R)$ whose induced t-structure $(\Ucal,\Vcal)$ corresponds to an sp-filtration $\Phi$. Then:

\begin{enumerate}[(i)]
\item If $(\Ucal,\Vcal)$ is restrictable then $\Hom_{\D(R)}(C^\varkappa,C)$ is flat as an $R$-module for any cardinal $\varkappa$.

\item If $\Hom_{\D(R)}(C^\varkappa,C)$ is flat as an $R$-module for any cardinal $\varkappa$ then $\Phi$ satisfies the weak Cousin condition.
\end{enumerate}

\noindent As a consequence, if $R$ is CM-excellent (see \cref{s:cmexc}) then $\Hom_{\D(R)}(C^\varkappa,C)$ is flat as an $R$-module for any cardinal $\varkappa$ if and only if $(\Ucal,\Vcal)$ is restrictable.
\end{prop}
\begin{proof}\leavevmode%
\subparagraph{(i)} By \cref{restrictable}, for any ring quotient $\lambda\colon R \to R/I$ the induced t-structure $(\lambda^* \Ucal,\lambda^! \Vcal)$ is restrictable in $\D(R/I)$. For any $Y \in \D(R/I)$ we have the adjunction isomorphism $\RHom_{R/I}(Y,\lambda^! C) \cong \RHom_R(\lambda_* Y,C)$, which shows that $(\lambda^* \Ucal,\lambda^! \Vcal) = (\Perp{\leq 0}\!\lambda^! C,\Perp{>0}\!\lambda^! C)$ is the cosilting t-structure in $\D(R/I)$ induced by the cosilting complex $\lambda^! C = \RHom_R(R/I,C)$. Combined with \cref{restrictable-props}(i), we see that $\lambda^! C$ is cotilting in $\D(R/I)$. Then the claim follows by \cref{cotilting-flat}.

\subparagraph{(ii)} By \cref{cotilting-flat}, we have for any ideal $I$ of $R$ that the cosilting complex $\lambda^! C$ in $\D(R/I)$ is cotilting, where $\lambda\colon R \to R/I$ si the quotient morphism. Towards contradiction, let $n \in \Zbb$ and $\pp \subsetneq \qq$ be a minimal inclusion of primes such that $\qq \in \Phi(n)$ but $\pp \not\in \Phi(n-1)$. The weak Cousin condition is clearly a local property, which together with \cref{cotilting-coloc} allows us to pass to the localization $R_\qq$. We thus assume without loss of generality that $R$ is local with the maximal ideal $\qq = \mm$. By the assumption, $\RHom_R(R/\pp,C)$ is a cotilting complex in the derived category $\D(R/\pp)$ of the one-dimensional local domain $R/\pp$. The corresponding sp-filtration $\lambda\Phi$ of $\Spec{R/\pp}$ satisfies by the construction $\lambda\Phi(n-1) = \lambda\Phi(n) = \{\overline{\mm}\}$, where $\overline{\mm} = \mm/\pp$ is the maximal ideal of $R/\pp$. Since $\lambda\Phi$ is non-degenerate and $R/\pp$ is one-dimensional, it follows that $\lambda\Phi$ is a slice filtration on $\Spec{R/\pp}$ (see \cref{ss:codimension}). On the other hand, $\lambda\Phi(n-1) = \lambda\Phi(n) = \{\overline{\mm}\}$ ensures that $\lambda\Phi$ is not a codimension filtration. Then $\RHom_R(R/\pp,C)$ cannot be cotilting by \cite[Proposition~6.10(2)]{HNS}, a contradiction.

The last claim follows from conditions $\mathrm{(i)}$, $\mathrm{(ii)}$, and \cref{restrictable-props}(iii).
\end{proof}

We come back to restrictable t-structures in \cref{s:cmexc}.

\subsection{Codimension sp-filtrations}%
\label{ss:codimension}%
As introduced in \cite{HNS}, an sp-filtration $\Phi$ on $\Spec R$ is a \newterm{slice filtration} if it is non-degenerate and $\dim(\Phi(n-1) \setminus \Phi(n)) \leq 0$ for each $n \in \Zbb$, that is, whenever $\pp, \qq \in \Phi(n-1) \setminus \Phi(n)$ are such that $\pp \subseteq \qq$ then $\pp = \qq$. The datum of an sp-filtration can equivalently be described by an order-preserving function $\ef\colon\Spec R \to \Zbb \cup \{-\infty,\infty\}$, such a function corresponds to an sp-filtration $\Phi_\ef$ defined by $\Phi_\ef(n) = \{\pp \in \Spec R \mid \ef(\pp)>n\}$, see \cite[\S 2.4]{HNS} or \cite{Tak22}. Note that $\Phi$ induces a non-degenerate t-structure if and only if the corresponding order-preserving function $\ef$ takes values in $\Zbb$, see \cite[Theorem~3.8]{AHH19}. It can be easily seen that: 

\begin{itemize}
\item $\Phi_\ef$ is slice if and only if $\ef\colon\Spec R \to \Zbb$ is strictly increasing;

\item $\Phi_\ef$ satisfies the weak Cousin condition if and only if $\ef\colon \Spec R \to \Zbb$ satisfies $\ef(\qq) \leq \ef(\pp) + \height(\qq/\pp)$ for any primes $\pp \subseteq \qq$.
\end{itemize}

If $\Phi$ is both a slice filtration and it satisfies the weak Cousin condition, we call it a \newterm{codimension filtration}. The corresponding function is a \newterm{codimension function}, that is, a function $\ed\colon \Spec R \to \Zbb$ such that for any $\pp \subseteq \qq$ in $\Spec R$ we have $\ed(\qq) - \ed(\pp) = \height \qq/\pp$. Let $T = \bigoplus_{\pp \in \Spec R}\RGamma_{\!\pp} R_\pp[\ed(\pp)]$ be the silting complex corresponding to $\Phi$, this explicit construction of $T$ is provided in \cite[\S 4]{HNS}. In particular, $T$ satisfies that $C = T^+$ is a cosilting complex inducing the cosilting t-structure $(\Ucal,\Vcal)$ which corresponds to $\Phi$.

The codimension function does not always exist --- when it does the ring has to be catenary, and the converse is true for local rings. If $\Spec R$ is connected then any two codimension functions on it differ only by adding a constant from $\Zbb$. For a local catenary ring $R$, the assignment $\pp \mapsto \dim(R) - \dim(R/\pp)$ is a codimension function, we call it the \newterm{standard codimension function}. If $\Spec R$ has a unique minimal element and a codimension function exists then the \newterm{height function} $\pp \mapsto \height(\pp)$ is a codimension function, and so in this situation any codimension function is of the form ${\height}+c$ for some $c \in \Zbb$. See e.g.\ \cite[\S 5.11, \S 10.105]{Stacks}.

In \cite[Theorem~7.5]{HNS}, it is proved that the silting complex associated to a codimension function on $\Spec R$ is always tilting whenever $R$ is a homomorphic image of a Cohen--Macaulay ring of finite Krull dimension. We show that this tilting complex is in fact product-complete. This generalizes, and is based on, the special case of the height function tilting module over a Cohen--Macaulay ring of Le~Gros and the first author \cite{HLG}.

\begin{prop}\label{prod-add}%
Let $R$ be a homomorphic image of a Cohen--Macaulay ring $S$ of finite Krull dimension and let $T$ be the tilting complex associated to a codimension function on $\Spec R$. Then $T$ is product-complete.
\end{prop}
\begin{proof}
Note first that by \cite[Remark~4.10]{HNS}, the choice of a codimension function does not matter, and by the assumption on $R$ a codimension function on $\Spec R$ always exists. If $R = S$ is already Cohen--Macaulay of finite Krull dimension, then by choosing the height function as the codimension function the statement is proved in \cite[Corollary~3.12]{HLG}. Now let $R = S/I$ for some ideal $I$ of $S$ and $\pi\colon S\to R$ be the projection map. Without loss of generality, we can assume that the codimension function $\ed\colon\Spec R \to \Zbb$ is given as $\ed(\pp) = \height_S(\pi^{-1}(\pp))$. Arguing as in \cite[\S 7]{HNS}, we can assume that $T \cong T' \otimes_S^\mathbf{L} R$ where $T'$ is the tilting $S$-module associated to the height function. By the above, we have that $\Add(T')$ is closed under products in $\Mod S$. Then we have a well-defined functor $- \otimes_S^\mathbf{L} R\colon\Add(T') \to \Add(T)$. Because $R$ is a finitely generated $S$-module, this functor preserves products existing in $\D^\bdd(S)$. Therefore, $\Add(T)$ is closed under products, arguing similarly as in the proof of \cite[Theorem~3.18]{HLG}.
\end{proof}

\begin{ex}
If $R$ is not a homomorphic image of a Cohen--Macaulay ring then \cref{prod-add} can fail. Indeed, by \cite[Proposition~4.5]{HRW01}, there is a local $2$-dimensional noetherian domain $R$ with field of quotients $Q$, whose generic formal fibre $Q \otimes_R \what{R}$ is not Cohen--Macaulay. The height function is a codimension function for this ring. Consider the induced silting complex $T = \bigoplus_{\pp \in \Spec R}\RGamma_{\!\pp} R_\pp[\height(\pp)]$. We claim that $\Add(T)$ is not product-closed. For each $\pp$ of height at most~$1$, $R_\pp$ is a one-dimensional domain, and thus is Cohen--Macaulay, so that $\RGamma_{\!\pp} R_\pp[\height(\pp)]$ is isomorphic to a module in degree zero. On the other hand, $R$ is not Cohen--Macaulay and so if $\mm$ is the maximal ideal we have that $\RGamma_{\!\mm} R [2]$ has cohomology non-vanishing in degree $-1$. By the assumption on $R$, it is not a homomorphic image of a Cohen--Macaulay ring (see \cref{s:cmexc}). In particular, $R$ is not a generalized Cohen--Macaulay ring (see \cite[\S 4]{Sch83}), which means that there is $i<\dim(R)=2$ such that $H^i\RGamma_{\!\mm} R$ is not annihilated by any single power of $\mm$. Since $R$ is a domain, it follows that the last sentence applies to $H^{1}\RGamma_{\!\mm} R = H^{-1}\RGamma_{\!\mm} R [2]$. It follows that $H^{-1}\RGamma_{\!\mm} R [2]^\varkappa$ is not supported on $V(\mm)$ whenever $\varkappa$ is an infinite cardinal. But then the product $\RGamma_{\!\mm} R [2]^\varkappa$ does not belong to $\Add(T)$, as for any $X \in \Add(T)$ we clearly have by the previous discussion that $H^{-1}X$ is supported on $V(\mm)$.
\end{ex}

\section{CM-excellent rings and restrictable t-structures}%
\label{s:cmexc}%

Following Kawasaki \cite{Kaw08} and \v Cesnavi\v cius \cite{Ces21}, $R$ is called \newterm{CM-excellent} if the following three conditions hold:

\begin{enumerate}[(1)]
\item $R$ is universally catenary;

\item all formal fibres of each local ring $R_\pp$ are Cohen--Macaulay;

\item $\CM A$ is an open subset of $\Spec A$ for any commutative finitely generated $R$-algebra $A$.
\end{enumerate}

\noindent Any Cohen--Macaulay ring and any ring admitting a classical dualizing complex is CM-excellent. By \cite[Remark~2.8]{Tak22}, condition (3) is equivalent to an a priori weaker condition:

\begin{enumerate}
 \item[(3')] $\CM A$ is an open subset of $\Spec A$ for any commutative module-finite $R$-algebra $A$.
\end{enumerate}

\noindent For a local ring $R$ to be CM-excellent, it is by \cite{Kaw02,Kaw08} enough to check (1) and a weakening of~(2):

\begin{enumerate}
 \item[(2')] all formal fibres of $R$ are Cohen--Macaulay.
\end{enumerate}

\noindent The following is a deep theorem of Kawasaki, showing a tight connection between CM-excellent rings and homomorphic images of Cohen--Macaulay rings. This should be seen as analogous to another Kawasaki's result from \cite{Kaw02}, the celebrated solution to Sharp's conjecture, characterizing rings with classical dualizing complexes as homomorphic images of finite-dimensional Gorenstein rings.

\begin{theorem}[{\cite[Theorem~1.3]{Kaw08}, \cite[Corollary~1.4]{Kaw02}}]\label{Kaw}%
The following are equivalent for a commutative noetherian ring:

\begin{enumerate}[(a)]
\item $R$ is a homomorphic image of a Cohen--Macaulay ring;

\item $R$ is CM-excellent and admits a codimension function on $\Spec R$. 
\end{enumerate}

\noindent In particular, if $R$ is local then \textup{(a)} holds if and only if $R$ is CM-excellent. Furthermore, if $R$ is of finite Krull dimension and \textup{(b)} holds then $R$ is a homomorphic image of a Cohen--Macaulay ring of finite Krull dimension.
\end{theorem}
\begin{proof}
The ``Furthermore'' part follows by Kawasaki's construction, see \cite[p.~123]{Kaw02}, \cite[p.~2738]{Kaw08}, and \cite[Theorem~15.7]{Mat89}.
\end{proof}

Let $R$ be a commutative noetherian ring of finite Krull dimension and let $\ed$ be a codimension function on $\Spec R$. As in the previous section, let $(\Ucal,\Vcal)$ be the induced compactly generated t-structure, let $T= \bigoplus_{\pp \in \Spec R}\RGamma_{\!\pp} R_\pp [\ed(\pp)]$ and $C=T^+$ be the induced silting and the character dual cosilting complex. We let $\Hcal_\mathrm{CM}^\ed$ denote the heart of the t-structure $(\Ucal,\Vcal)$. We call this heart the \newterm{Cohen--Macaulay heart}, note that it does not depend on the choice of the codimension function up to categorical equivalence, \cite[Remark~4.10]{HNS}. Therefore, when not concerned with the particular way the Cohen--Macaulay heart is embedded into $\D(R)$, we can denote it simply as $\Hcal_\mathrm{CM}$.

\begin{prop}\label{noeth-coh-cop}%
 If $R$ is CM-excellent then $C$ is cotilting and $\Hcal_\mathrm{CM}$ is locally coherent and locally coperfect.
\end{prop}
\begin{proof}
We know that $T$ is product-complete by \cref{prod-add}. Then we infer from \cref{loc-coh-cop} that $T$ is decent, so that $C$ is cotilting, and that $\Hcal_\mathrm{CM}$ is locally coherent and locally coperfect.
\end{proof}

\begin{lemma}\label{locus}%
 Let $\ed$ be a codimension function on $\Spec R$. Assume that the compactly generated t-structure $(\Ucal,\Vcal)$ corresponding to $\ed$ is restrictable. Then the Cohen--Macaulay locus of any commutative module-finite $R$-algebra $A$ is open.
\end{lemma}
\begin{proof}
By \cref{restrictable}, it is enough to prove that $\CM R$ is open. Indeed, let $\Phi$ be a codimension filtration on $\Spec R$ and $(\Ucal,\Vcal)$ the corresponding t-structure. Consider the induced sp-filtration $\lambda\Phi$, where $\lambda\colon R \to A$ is the algebra map, and $(\lambda^* \Ucal,\lambda^! \Vcal)$ the induced t-structure in $\D(A)$. By \cite[Lemma~7.1]{HNS}, $\lambda\Phi$ is a slice filtration on $\Spec A$. Since $(\lambda^* \Ucal,\lambda^! \Vcal)$ is a restrictable t-structure by \cref{restrictable}, it follows that $\lambda\Phi$ satisfies the weak Cousin condition by \cref{restrictable-props}, and thus $\lambda\Phi$ is a codimension filtration.

Let us first assume that $\ed = \height$. Then $\Vcal = \{X \in \D(R) \mid \depth X_\pp \geq \height \pp\ \forall \pp \in \Spec R\}$, see \cite[\S 2.3, 2.4]{HNS}. It follows that $R[0] \in \Vcal$ if and only if $R$ is a Cohen--Macaulay ring. Consider the approximation triangle $U \to R[0] \to V\buildrel+\over\to$ of $R[0]$ with respect to the t-structure $(\Ucal,\Vcal)$. For each $\pp \in \Spec R$, the localized triangle $U_\pp \to R_\pp[0] \to V_\pp\buildrel+\over\to$ is the approximation triangle of $R_\pp[0]$ with respect to the t-structure $(\Ucal_\pp,\Vcal_\pp)$ in $\D(R_\pp)$, see \cite[Lemma~3.4]{HHZ21}, and note that this latter t-structure is compactly generated and corresponds to the height function on $\Spec {R_\pp}$. It follows that $U_\pp = 0$ if and only if $R_\pp$ is a Cohen--Macaulay ring. Therefore, $\CM R = \Spec R \setminus \Supp(U)$. Since $(\Ucal,\Vcal)$ is assumed to be restrictable, $U \in \D^\bdd_\fg(R)$, and therefore $\Supp(U)$ is a closed subset of $\Spec R$.

Now let $\ed$ be a general codimension function. By the Nagata criterion, $\CM R$ is open provided that $\CM {R/\pp}$ is open for every $\pp \in \Spec R$, see \cite[Theorem~24.5]{Mat89}. For each $\pp \in \Spec R$, the restriction of $\ed$ to $V(\pp)$ is a codimension function $\ed_{\pp}$ for $\Spec {R/\pp}$. Since $\Spec {R/\pp}$ has a unique minimal element, the codimension function $\ed_{\pp}$ is equal to a height function up to some additive constant. By \cref{restrictable}, the t-structure $(\Ucal,\Vcal)$ induces a restrictable t-structure $(\lambda^* \Ucal,\lambda^! \Vcal)$ in $\D(R/\pp)$, and this latter t-structure is induced by the codimension function $\ed_{\pp}$. Since $\ed_{\pp} = {\height}+c$ for some constant $c \in \Zbb$, the previous paragraph shows that $\CM {R/\pp}$ is open in $\Spec {R/\pp}$.
\end{proof}

We are now ready to characterize when the t-structure $(\Ucal,\Vcal)$ is restrictable.

\begin{theorem}\label{restrict-char}%
Let $R$ be a commutative noetherian ring of finite Krull dimension such that there is a codimension function $\ed$ on $\Spec R$. Let $(\Ucal,\Vcal)$ be the compactly generated t-structure induced by $\ed$ and let $C$ be the corresponding cosilting complex. The following are equivalent:

\begin{enumerate}[(a)]
\item $R$ is a CM-excellent ring;

\item $(\Ucal,\Vcal)$ is restrictable;

\item $C$ is cotilting and $\Hom_R(C^\varkappa,C)$ is a flat $R$-module for each $\varkappa$.
\end{enumerate}
\end{theorem}
\begin{proof}\leavevmode
\subparagraph{$\mathrm{(a)}\Rightarrow\mathrm{(b)}$} Since $C$ is cotilting and the cotilting heart is locally coherent by \cref{noeth-coh-cop}, the restrictability follows from \cref{restrictable-props}.
 
\subparagraph{$\mathrm{(b)}\Rightarrow\mathrm{(c)}$} This is \cref{restrictable-flat}(i).

\subparagraph{$\mathrm{(c)}\Rightarrow\mathrm{(a)}$} It follows by \cite[Proposition~7.15, Theorem~7.18]{HNS} that $R$ is universally catenary and all formal fibres of all stalk rings $R_\pp$ are Cohen--Macaulay. By combining \cref{restrictable} with \cref{locus}, the Cohen--Macaulay locus of any module-finite $R$-algebra $A$ is open, and thus $R$ is CM-excellent.
\end{proof}

\begin{rmk}
\cite[Question~7.8]{HNS} asks whether the cosilting complex induced by a codimension function is cotilting if and only if $R$ is CM-excellent. As a particular answer, the equivalence between (a) and (c) of \cref{restrict-char} is proved in \cite[Theorem~7.19]{HNS} for $R$ local. Our \cref{restrict-char} thus improves our knowledge by removing the locality assumption and introducing the restrictability condition (b) into the picture.
\end{rmk}

The following consequence shows that the recent Takahashi's generalization of \cite[Theorem~6.9]{ATJLS10} from rings with classical dualizing complexes to CM-excellent rings is the maximal generality, at least when assuming the existence of a codimension function.

\begin{cor}\label{RTthm}%
Let $R$ be a commutative noetherian ring of finite Krull dimension with a codimension function. The following are equivalent:
 
\begin{enumerate}[(a)]
\item any compactly generated t-structure satisfying the weak Cousin condition is restrictable;

\item $R$ is CM-excellent.
\end{enumerate}
\end{cor}
\begin{proof}\leavevmode
\subparagraph{$\mathrm{(a)}\Rightarrow\mathrm{(b)}$} If $R$ is not CM-excellent, the t-structure induced by any codimension function is not restrictable by \cref{restrict-char}, a contradiction.

\subparagraph{$\mathrm{(b)}\Rightarrow\mathrm{(a)}$} Proved in \cite{Tak22}.
\end{proof}

The following example shows that the assumption of having a codimension function cannot be simply removed from \cref{RTthm}.

\begin{ex}
Let $R$ be a non-catenary (thus, not CM-excellent) local normal $3$-dimensional domain, see \cite{Hei82} or \cite[Example~2.15]{Nis12}. We claim that the condition (1) of \cref{RTthm} holds for $R$. 

 Let $\Phi$ be an sp-filtration satisfying the weak Cousin condition and $\ef$ the corresponding order-preserving function $\Spec R \to \Zbb$ satisfying $\ef(\qq) \leq \ef(\pp) + \height(\qq/\pp)$ for any $\pp \subseteq \qq$. Let $(\Ucal,\Vcal)$ be the corresponding compactly generated t-structure. By shifting, we can assume that $\ef(0) = 0$. If $\ef(\pp) \leq 1$ for all $\pp \in \Spec R$ then $(\Ucal,\Vcal)$ is Happel--Reiten--Smal\o, and thus restrictable. Since $\dim R = 3$ and $R$ is not catenary, we have $\ef(\pp) \leq 2$ for any $\pp \in \Spec R$ by the weak Cousin condition. We thus have $\Phi(2) = \emptyset$ and we can assume $\Phi(1) \neq \emptyset$. By the weak Cousin condition again, it follows that any $\pp \in \Phi(i)$ has $\height(\pp) \geq i+1$ for $i=0,1$.

By \cite[Theorem~4.4]{ATJLS10}, to check that $(\Ucal,\Vcal)$ is restrictable, it is enough to show that $H^1\RGamma_{\!\Phi(1)}(X)$ is finitely generated for $X \in \Vcal' \cap \D^\bdd_\fg(R)$, where $\Vcal' = \{X \in \D^{\geq 0} \mid \Supp(H^0(X)) \subseteq \Phi(0)\}$. Considering the soft truncation triangle $H^0(X) \to X \to \tau^{>0}X\buildrel+\over\to$ and applying $\RGamma_{\!\Phi(1)}$ yields an exact sequence 
\[
	0 \longrightarrow H^1\RGamma_{\!\Phi(1)}(H^0(X))\longrightarrow
		\RGamma_{\!\Phi(1)}(X)\longrightarrow \RGamma_{\!\Phi(1)}(\tau^{>0}X),
\]
where $\RGamma_{\!\Phi(1)}(\tau^{>0}X)$ is finitely generated as it is isomorphic to $\Gamma_{\!\Phi(1)}(H^1(X))$. We reduced the task to showing that $H^1\RGamma_{\!\Phi(1)}(M)$ is finitely generated for any finitely generated module $M$ supported on $\Phi(0)$. Any such $M$ is torsion-free over the domain $R$, and so there is a short exact sequence $0 \to M \to R^k \to N \to 0$ for some $k>0$, see \cite[Lemma~16.1]{GT12}. Then we have an exact sequence $\Gamma_{\!\Phi(1)}N \to H^1\RGamma_{\!\Phi(1)}M \to H^1\RGamma_{\!\Phi(1)}R^k$. As again $\Gamma_{\!\Phi(1)}N$ is always finitely generated, this reduced the task to $M = R$. But since $R$ is normal, Serre's criterion \cite[Lemma~10.157.4]{Stacks} yields that $\depth R_\qq \geq 2$ for each $\qq \in \Phi(1)$, and thus $H^1\RGamma_{\!\Phi(1)}R = 0$.
\end{ex}

\section{CM-excellent rings and derived equivalences}%
\label{s:cmexc2}%
Now the scene is set for us to characterize homomorphic images of Cohen--Macaulay rings in terms of derived equivalences and dualities. Before that, we need to record a localization property of product-complete tilting complexes.

\begin{lemma}\label{product-complete-tilting-loc}%
Let $R$ be a commutative ring and $T \in \D(R)$ a product-complete tilting complex. Then for any prime ideal $\pp \in \Spec R$, $T_\pp = T \otimes_R R_\pp$ is a product-complete object in $\D(R_\pp)$.
\end{lemma}
\begin{proof}
By \cref{product-complete-tilting}, $\Add(T)$ is a definable subcategory of $\D(R)$. In particular, $T_\pp \in \Add(T)$, because $T_\pp = T \otimes_R R_\pp$ is a direct limit (and thus also a directed homotopy colimit) of copies of $T$. Then $\Add(T_\pp) \subseteq \Add(T)$, and in fact, $\Add(T_\pp) = \Add(T) \cap \D(R_\pp)$. It follows that $\Add(T_\pp)$ is closed under products both as a subcategory of $\D(R)$ or $\D(R_\pp)$.
\end{proof}

\begin{theorem}\label{char-eq}%
Let $R$ be a commutative noetherian ring. Consider the following conditions:

\begin{enumerate}[(a)]
\item $R$ is a homomorphic image of a Cohen--Macaulay ring of finite Krull dimension.
 
\item There is a locally coherent and locally coperfect abelian category $\Ccal$ together with a triangle equivalence $\D^\bdd(R) \cong \D^\bdd(\Ccal)$.
 
\item There is an artinian abelian category $\Acal$ together with a triangle equivalence $\D^\bdd_\fg(R) \cong \D^\bdd(\Acal)$.
 
\item There is a noetherian abelian category $\Bcal$ together with a triangle equivalence $\D^\bdd_\fg(R) \cong \D^\bdd(\Bcal)^\op$.
\end{enumerate} 

\noindent Then $\mathrm{(a)\Leftrightarrow (b)\Rightarrow (c)\Leftrightarrow (d)}$, and all the conditions are equivalent if $R$ is of finite Krull dimension.

Furthermore, if \textup{(a)}, \textup{(b)} hold then $\Ccal \cong \Hcal_\mathrm{CM}$, $\Acal \cong \fp{\Hcal_\mathrm{CM}}$, and $\Bcal \cong \fp{\Hcal_\mathrm{CM}}^\op$.
\end{theorem}
\begin{proof}\leavevmode
\subparagraph{$\mathrm{(a)}\Rightarrow\mathrm{(b)}$} Already proved in \cref{prod-add} and \cref{loc-coh-cop}.

\subparagraph{$\mathrm{(b)}\Rightarrow\mathrm{(a)}$} Let $W$ be an injective cogenerator of $\Ccal$ (see \cref{ss:loccoh}). Let us denote the derived equivalence as $E\colon\D^\bdd(\Ccal) \toeq \D^\bdd(R)$ and put $C = E(W)$. By (the proof of) \cref{tilting-bounded}, $C \in \D(R)$ is a cotilting complex. By \cite[Corollary~2.14]{HN21}, the cotilting t-structure $(\Ucal,\Vcal)$ induced by $C$ is compactly generated, and note that this t-structure corresponds to the standard t-structure in $\D^\bdd(\Ccal)$ under the equivalence. Since $C$ is a cotilting complex, $(\Ucal,\Vcal)$ is intermediate. Let $\Phi$ be the sp-filtration on $\Spec R$ corresponding to $(\Ucal,\Vcal)$. Since $\Hcal_C$ is locally coherent and $C$ is cotilting, $(\Ucal,\Vcal)$ is a restrictable t-structure by \cref{restrictable-props} and $\Phi$ satisfies the weak Cousin condition. It remains to show that $\Phi$ is a slice filtration, as then $\Phi$ is induced by a codimension function $\ed$ on $\Spec R$. Then $R$ admits a codimension function, which together with intermediacy of $(\Ucal,\Vcal)$ implies that $\dim(R)<\infty$. By \cref{restrict-char}, $R$ is CM-excellent, and so it is a homomorphic image of a finite-dimensional Cohen--Macaulay ring by \cref{Kaw}. 

Towards contradiction, let $\pp \subsetneq \qq$ be primes such that $\pp,\qq \in \Phi(n-1) \setminus \Phi(n)$. It follows that $\Mod{(R/\pp)_\qq}[-n] \subseteq \Hcal_C$. Indeed, let $M \in \Mod{(R/\pp)_\qq}$ and recall the description of the t-structure $(\Ucal,\Vcal)$ of \cref{ss:ajs} in terms of $\Phi$. Since $\Supp(M) \subseteq \Phi(n-1)$, $M[-n]$ belongs to $\Ucal[-1]$. On the other hand, $\RGamma_{\Phi(k)}(M)$ is either isomorphic to $M$ for $k<n$ or vanishes for $k \geq n$, and thus $M[-n] \in \Vcal$. We showed that $M[-n] \in \Ucal[-1] \cap \Vcal = \Hcal_C$.

Let $T$ be a tilting complex such that $T^+$ is equivalent to $C$. By \cref{loc-coh-cop}, $T$ is product-complete. It follows by \cref{product-complete-tilting-loc} that $T_\qq$ is a product-complete tilting complex in $\D(R_\qq)$ for any $\qq \in \Spec R$. Therefore, we can assume without loss of generality that $R$ is local with maximal ideal $\qq$. Then $\Mod R/\pp[-n] \subseteq \Hcal_C$ by the above, and so the lattice of ideals of the ring $R/\pp$ embeds into the lattice of finitely presentable subobjects of $(R/\pp)[-n]$ in $\Hcal_C$. But since $\dim(R/\pp)>0$, $R/\pp$ is not a perfect ring and thus it follows that $(R/\pp)[-n]$ is not a coperfect object in $\Hcal_C$.

\subparagraph{$\mathrm{(b)}\Rightarrow\mathrm{(c)}$} Recall that if $\Ccal$ is a locally coherent abelian category then $\Ccal$ is locally coperfect if and only if $\fp{\Ccal}$ artinian. We know that $\mathrm{(b)}\Rightarrow \mathrm{(c)}$ because the cotilting derived equivalence $\D^\bdd(R) \cong \D^\bdd(\Hcal_C)$ restricts to $\D^\bdd_\fg(R) \cong \D^\bdd(\fp{\Hcal_C})$ by \cite[Lemma~3.11]{HP}. 

\subparagraph{$\mathrm{(c)} \Rightarrow \mathrm{(b)}$ under $\dim(R)<\infty$} Assume $\mathrm{(c)}$, then the standard t-structure on $\D^\bdd(\Acal)$ is sent to a t-structure in $\D^\bdd_\fg(R)$, and such a t-structure uniquely extends to a compactly generated t-structure $(\Ucal,\Vcal)$ in $\D(R)$, \cref{restrictable-props}. The heart $\Hcal$ of $(\Ucal,\Vcal)$ satisfies $\fp{\Hcal} \cong \Acal$ by \cite[Theorem~6.3]{Sao17}, and so $\Hcal$ is a locally coherent and locally coperfect Grothendieck category. Since $(\Ucal,\Vcal)$ is restrictable, it is cotilting by \cref{restrictable-props}. Finally, $(\Ucal,\Vcal)$ is intermediate because it arises from a weak Cousin filtration and $\dim(R)<\infty$, so we have a triangle equivalence $\D^\bdd(R) \cong \D^\bdd(\Hcal)$ by \cref{tilting-bounded}, cf.\ \cref{ss:ajs}.

\subparagraph{$\mathrm{(c)}\Leftrightarrow\mathrm{(d)}$} Clear, as $\D^\bdd(\Acal) \cong \D^\bdd(\Acal^\op)^\op$.

Finally, the proof of $\mathrm{(b)}\Rightarrow \mathrm{(a)}$ shows that $\Ccal \cong \Hcal_\mathrm{CM}$ and the proof of $\mathrm{(b)}\Leftrightarrow \mathrm{(c)}$ shows that $\Acal \cong \fp{\Ccal} \cong \fp{\Hcal_\mathrm{CM}}$.
\end{proof}

\begin{rmk}
 There are rings $R$ of infinite Krull dimension such that $(\mathrm{c})$ of \cref{char-eq} holds. Indeed, we can choose $R$ with a (strongly pointwise) dualizing complex, so that $\D^\bdd_\fg(R) \cong \D^\bdd_\fg(R)^\op$, e.g.\ \cite{Nag56}, cf.\ \cite{Nee10}. The proof of $(\mathrm{c})\Rightarrow (\mathrm{b})$ of \cref{char-eq} breaks because the compactly generated t-structure $(\Ucal,\Vcal)$ which extends the t-structure induced in $\D^\bdd_\fg(R)$ by the duality is not intermediate.
\end{rmk}

\begin{cor}\label{prod-com-tilting-ex}%
 Let $R$ be a commutative noetherian ring. The following are equivalent:
 
\begin{enumerate}[(a)]
\item there is a product-complete tilting complex $T \in \D(R)$; 

\item $R$ is a homomorphic image of a Cohen--Macaulay ring of finite Krull dimension.
\end{enumerate}
 
\noindent In addition, if \textup{(a)} holds then $T$ is induced by a codimension function on $\Spec R$. In particular, $T$ is unique up to equivalence and a choice of shift constant on each connected component of $\Spec R$.
\end{cor}
\begin{proof}
For the equivalence of $\mathrm{(a)}$ and $\mathrm{(b)}$, combine \cref{char-eq} with \cref{loc-coh-cop} and \cref{prod-add}. Furthermore, if $T$ is a product-complete tilting complex then the corresponding sp-filtration has to be induced by a codimension function by \cref{char-eq} and \cref{loc-coh-cop}, which also yields the uniqueness statement. 
\end{proof}

\begin{rmk}
 In \cite[Theorem~3.18]{HLG}, it is shown that a commutative noetherian ring admits a product-complete tilting \newterm{module} if and only if $R$ is Cohen--Macaulay of finite Krull dimension, and then such a tilting module is unique up to equivalence. \cref{prod-com-tilting-ex} can be seen as a derived version of this result.
\end{rmk}

\section{Gorenstein complexes}%
\label{s:gor}%
An object $D \in \D^\bdd_\fg(R)$ is a \newterm{dualizing complex} if the functor $\RHom_R(-,D)$ yields an equivalence $\D^\bdd_\fg(R) \toeq \D^\bdd_\fg(R)^\op$. Equivalently, $\RHom_R(X,D) \in \D^\bdd_\fg(R)$ and the canonical map $X \to \RHom_R(\RHom_R(X,D),D)$ is an isomorphism for all $X \in \D^\bdd_\fg(R)$. A dualizing complex is called \newterm{classical} if it is of finite injective dimension --- this occurs precisely if $\dim(R)<\infty$. To any dualizing complex $D$, the function $\ed_D\colon\Spec R \to \Zbb$, defined by setting $\ed_D(\pp)$ to be the unique integer such that $\Hom_{\D(R)}(\kappa(\pp),D_\pp[\ed_D(\pp)]) \neq 0$, is a codimension function. If $(R,\mm)$ is local, we call $D$ a \newterm{normalized dualizing complex} if $\ed_D(\mm) = \dim(R)$; a normalized dualizing complex is essentially unique.

\begin{rmk}
What we call a classical dualizing complex is traditionally called just a dualizing complex. We follow the modern terminology of Neeman \cite{Nee10}.
\end{rmk}

Following Grothendieck and Hartshorne \cite{Har66}, a complex $G \in \D^\bdd_\fg(R)$ is a \newterm{Cohen--Macaulay complex} (with respect to a codimension function $\ed\colon\Spec R \to \Zbb$) if for each $\pp \in \Spec R$ we have $H^i\RGamma_{\!\pp}G_\pp=0$ for all $i\neq\ed(\pp)$. These are precisely the complexes which are quasi-isomorphic to their Cousin complex, see \cite[\S IV.3]{Har66}. We call $G$ a \newterm{Gorenstein complex} if, in addition, $H^{\ed(\pp)}\RGamma_{\!\pp}G_\pp$ is an injective $R$-module. In this case, the Cousin complex yields an injective resolution of $G$.

We gather some facts about Gorenstein complexes first.

\begin{lemma}\label{Gor-prop}%
 Let $G$ be a Gorenstein complex in $\D^\bdd_\fg(R)$ and let $S = \End_{\D(R)}(G)$ be its endomorphism ring. Then:
 
\begin{enumerate}[(i)]
\item a dualizing complex is a Gorenstein complex;

\item $\Hom_{\D(R)}(G,G[i]) = 0$ for any $i \neq 0$;

\item for any $\pp \in \Spec R$, $G_\pp$ is a Gorenstein complex in $\D^{\bdd}_\fg(R_\pp)$;

\item if $R$ is local then $\what{G} = G \otimes_R \what{R}$ is a Gorenstein complex over the completion $\what{R}$;

\item if $R$ is local then there is $k>0$ such that $\what{G} \cong \smash{D_{\what{R}}^k}[\dim(R)-\ed(\mm)]$ where $D_{\what{R}}$ is the normalized dualizing complex over $\what{R}$ and $\ed$ is the codimension function associated to $G$;

\item $S$ is a module-finite and projective $R$-algebra and there is $k>0$ such that $\what{S} = S \otimes_R \what{R}$ is isomorphic as an $\what{R}$-algebra to $M_{k}(\what{R})$, the ring of $k \times k$ matrices over $\what{R}$.
\end{enumerate}
\end{lemma}
\begin{proof}\leavevmode
\subparagraph{(i)} \cite[p.~287]{Har66}.
\subparagraph{(ii)} \cite[Proposition~6.2.5(a)]{NS}.
\subparagraph{(iii)} Clear from definition.
\subparagraph{(iv)} \cite[Remark~6.3.5]{NS}.
\subparagraph{(v)} Combine (i), (iv), and \cite[Theorem~6.2.6]{NS18}.
\subparagraph{(vi)} The first statement follows from \cite[Proposition~6.2.5(a)]{NS}. Furthermore, we have:
\begin{align*}
	\what{S} = S \otimes_R \what{R} &= \End_{\D(R)}(G) \otimes_R \what{R} \cr
	&\cong \End_{\D(\what{R})}(\what{G}) =
		\End_{\D(\what{R})}(D_{\what{R}}^k) = M_{k}(\what{R}). \qedhere
\end{align*}
\end{proof}

\begin{rmk}
There are more results available in the literature about the special case of a Gorenstein module. Sharp \cite{Sh69} introduced the \newterm{Gorenstein modules}, these are precisely those finitely generated $R$-modules which happen to be Gorenstein complexes as objects of $\D(R)$. If a Gorenstein module exists then $R$ is Cohen--Macaulay with Gorenstein formal fibres \cite{Sha79}. In particular, not every Cohen--Macaulay ring admits a Gorenstein complex. In \cref{Gor-eq}, we shall prove that the existence of a Gorenstein complex implies that $R$ is CM-excellent.
\end{rmk}

\begin{lemma}\label{heart-cm-complexes}%
Let $R$ be a commutative noetherian ring and $\ed$ a codimension function on $\Spec R$. An object $X \in \D^\bdd_\fg(R)$ is a Cohen--Macaulay complex with respect to $\ed$ if and only if it belongs to $\Hcal_\mathrm{CM}^\ed$.
\end{lemma}
\begin{proof}
Let $(\Ucal,\Vcal)$ be the compactly generated t-structure corresponding to $\ed$, so that its heart is $\Hcal_\mathrm{CM}^\ed$. By the description discussed in \cref{s:comnoeth}, $X\in\Vcal$ if and only if $\RGamma_{\!\pp}X_\pp \in \D^{\geq \ed(\pp)}$ for all $\pp\in\Spec R$. Therefore, we can further assume that $X$ satisfies both these conditions, and it remains to show that $X\in\Ucal[-1]$ if and only if $X$ is a Cohen--Macaulay complex.

If $X$ is Cohen--Macaulay then we can assume it being represented by its Cousin complex. By the construction of the Cousin complex, the $i$-th component of $X$ is supported on primes $\pp$ with $\ed(\pp) \geq i$. Then the same is true for $H^i(X)$, and so $X \in \Ucal[-1]$. Assume conversely that $X \in \Ucal[-1]$ and let $\pp \in \Spec R$ and let us show that $\RGamma_{\!\pp}X_\pp \in \D^{\leq \ed(\pp)}$. Since $X$ if a finite extension of its cohomology stalks $H^i(X)[-i], i \in \Zbb$, we can assume that $X = M[-i]$, where $M$ is an $R$-module supported on primes $\pp$ with $\ed(\pp) \geq i$. Passing to localization, we can assume that $R$ is local with maximal ideal $\pp$, and also that $\ed(\pp) \geq i$. It follows that $\dim(\Supp(M)) \leq \ed(\pp) - i$, and thus the vanishing theorem of Grothendieck \cite[Lemma~51.4.7, Proposition~20.20.7]{Stacks} implies $\RGamma_{\!\pp}X_\pp \in \D^{\leq \ed(\pp)}$.
\end{proof}

\begin{rmk}
Assume that $R$ is a homomorphic image of a finite-dimensional Cohen--Macaulay ring. By \cref{restrict-char}, we know that any codimension t-structure then restricts to $\D^\bdd_\fg(R)$. Combined with \cite[Theorem~6.3]{Sao17}, \cref{heart-cm-complexes} then shows that the Cohen--Macaulay complexes with respect to $\ed$ in $\D^\bdd_\fg(R)$ are precisely the objects in $\fp{\Hcal_\mathrm{CM}^\ed}$, the heart of the restricted t-structure. This can be seen as an affine version of \cite[Theorem~6.2]{YZ06}, but valid in the absence of a dualizing complex.
\end{rmk}

Assume that $\Spec R$ admits a codimension function $\ed$. The codimension function on $\Spec R$ restricts to a codimension function $\ed_\pp$ on $\Spec {R_\pp}$. Consider the Cohen--Macaulay heart $\Hcal_\mathrm{CM}^{\ed}$ in $\D(R)$ as well as the Cohen--Macaulay heart $\Hcal_\mathrm{CM}^{\ed_\pp}$ in $\D(R_\pp)$. Then $\Hcal_\mathrm{CM}^{\ed_\pp} = \Hcal_\mathrm{CM}^{\ed} \cap \D(R_\pp)$, and the inclusion $\Hcal_\mathrm{CM}^{\ed_\pp} \subseteq \Hcal_\mathrm{CM}^{\ed}$ has a left adjoint induced by the localization functor $R_\pp \otimes_R -$.

\begin{lemma}\label{fpinj-local}%
Let $R$ be a CM-excellent ring of finite Krull dimension with a codimension function. Then $R_\pp \otimes_R -$ induces a functor $\fp{\Hcal_\mathrm{CM}^{\ed}} \to \fp{\Hcal_\mathrm{CM}^{\ed_\pp}}$ which is essentially surjective up to direct summands.
\end{lemma}
\begin{proof}
 By the assumption, the t-structure corresponding to any codimension function is restrictable both in $\D(R)$ and $\D(R_\pp)$ by \cref{restrict-char}. We thus have $\fp{\Hcal_\mathrm{CM}^{\ed}} = \Hcal_\mathrm{CM}^{\ed} \cap \D^\bdd_\fg(R)$ and $\fp{\Hcal_\mathrm{CM}^{\ed_\pp}} = \Hcal_\mathrm{CM}^{\ed_\pp} \cap \D^\bdd_\fg(R_\pp)$. It follows that we have a well-defined functor $\fp{\Hcal_\mathrm{CM}^{\ed}} \to \fp{\Hcal_\mathrm{CM}^{\ed_\pp}}$. Let $X \in \fp{\Hcal_\mathrm{CM}^{\ed_\pp}}$. Consider $X$ as an object in $\Hcal_\mathrm{CM}^\ed$ and write $X = \varinjlim_{i \in I}F_i$ as a direct limit of finitely presentable objects of $\Hcal_\mathrm{CM}^\ed$. Since direct limits inside $\Hcal_\mathrm{CM}^\ed$ are computed as directed homotopy colimits, we have $X \cong R_\pp \otimes_R X = R_\pp \otimes_R \varinjlim_{i \in I}F_i \cong \varinjlim_{i \in I}(R_\pp \otimes_R F_i)$. The last direct limit can be viewed as computed in $\Hcal_\mathrm{CM}^{\ed_\pp}$, and each $R_\pp \otimes_R F_i$ is a finitely presentable object in $\Hcal_\mathrm{CM}^{\ed_\pp}$. It follows that there is $i \in I$ such that $X$ is a direct summand of $R_\pp \otimes_R F_i$.
\end{proof}

\begin{rmk}\label{fp-inj-cogen}
In what follows, we will consider the existence of an injective cogenerator of the category $\fp{\Ccal}$ of finitely presentable objects of a locally coherent category $\Ccal$. Note that this situation is more general then the existence of an injective cogenerator of the unrestricted category $\Ccal$ which happens to be finitely presentable. In fact, $W \in \fp{\Ccal}$ is an injective object of the category $\fp{\Ccal}$ if and only if $W$ is \newterm{fp-injective} as an object of $\Ccal$. Here, we call an object $G \in \Ccal$ \newterm{fp-injective} if $\Ext_\Ccal^1(F,G) = 0$ for all finitely presentable objects $F \in \Ccal$. Finally, the full subcategory of fp-injective objects of $\Ccal$ is equal to $\Def(W)$, the definable closure of an injective cogenerator $W$ in $\Hcal$ (cf.\ \cite[Example~5.11]{Lak20}). 
\end{rmk}

\begin{prop}\label{Gorenstein-fpinj}%
 Let $R$ be a CM-excellent ring of finite Krull dimension with a codimension function $\ed$. Then an object $G \in \D^\bdd_\fg(R)$ is a Gorenstein complex with respect to $\ed$ if and only if it is an injective object in $\fp {\Hcal_\mathrm{CM}^\ed}$.
 
As a consequence, the following are equivalent for $R$:
\begin{enumerate}[(a)]
\item there is a Gorenstein complex $G$ in $\D^\bdd_\fg(R)$;

\item $\fp {\Hcal_\mathrm{CM}}$ admits an injective cogenerator.
\end{enumerate}
\end{prop}
\begin{proof}
First, let $G$ be a Gorenstein complex and assume that $\ed$ is the codimension function associated to $G$. By \cref{heart-cm-complexes}, $G$ belongs to $\Hcal_\mathrm{CM}^\ed$. Since $(\Ucal,\Vcal)$ is a restrictable t-structure by \cref{restrict-char}, we have $\fp{\Hcal_\mathrm{CM}^\ed} = \Hcal_\mathrm{CM}^\ed \cap \D^\bdd_\fg(R)$ by \cite[Theorem~6.3]{Sao17}, and thus $G \in \fp{\Hcal_\mathrm{CM}^\ed}$. We claim that $G$ is an injective cogenerator in the latter abelian category. Let $F \in \fp{\Hcal_\mathrm{CM}^\ed}$, then 
$$\Ext_{\Hcal_\mathrm{CM}^\ed}^1(F,G) \cong \Hom_{\D(\Hcal_\mathrm{CM}^\ed)}(F,G[1]) \cong \Hom_{\D(R)}(F,G[1])$$ vanishes if and only if $$\Hom_{\D(R)}(F,G[1]) \otimes_R R_\mm \cong \Hom_{\D(R_\mm)}(F_\mm,G_\mm[1])$$ vanishes for each maximal ideal $\mm$ of $R$, the last isomorphism follows since $F \in \D^\bdd_\fg(R)$ and $G$ is up to quasi-isomorphism a bounded complex of injectives. Similarly, we can check the vanishing of $\Hom_{\Hcal_\mathrm{CM}^\ed}(F,G)$ locally. Since $G_\pp$ is a Gorenstein complex in $\D^\bdd_\fg(R_\pp)$ by \cref{Gor-prop} and $F_\pp$ belongs to $\fp{\Hcal_\mathrm{CM}^{\ed_\pp}}$ by \cref{fpinj-local}, the question reduces to $(R,\mm)$ being a local ring.
 
By shifting, we can assume that $\ed$ is the standard codimension function over the local ring $R$. By \cref{Gor-prop}, $\what{G}$ is a Gorenstein complex in $\D^\bdd_\fg(\what{R})$ and $\what{G} \cong D_{\what{R}}^k$ for some $k>0$, where $D_{\what{R}}$ is the normalized dualizing complex over $\what{R}$. It follows that there is a pure monomorphism $G \lhook\joinrel\to \what{G} \cong \smash{D_{\what{R}}^k}$ in $\D(R)$, and thus $G \in \Def(C)$, where $C = \prod_{\pp \in \Spec R}D_{\what{R_\pp}}[\height(\pp)-\ed(\pp)]$ is the cotilting complex corresponding to $\ed$, see \cite[\S 5]{HNS}. In view of \cref{fp-inj-cogen}, $G$ is an fp-injective object of $\Hcal_\mathrm{CM}^\ed$, and thus an injective object of $\fp{\Hcal_\mathrm{CM}^\ed}$. Finally, let us show that $\Hom_{\D(R)}(F,G)$ is non-zero for any non-zero object $F \in \fp{\Hcal_\mathrm{CM}^\ed}$. It suffices to show the non-vanishing of
\[
	\Hom_{\D(R)}(F,G) \otimes_R \what{R} \cong
		\Hom_{\D(\what{R})}(\what{F},\what{G}) =
		\Hom_{\D(\what{R})}(\what{F},D_{\what{R}}^k).
\]
Since $F \in \Hcal_\mathrm{CM}^\ed$, the vanishing of the last $\Hom$-module actually implies the vanishing of $\RHom_{\what{R}}(\what{F},D_{\what{R}})$, a contradiction with $\what{F} \neq 0$.

Let $C$ be the cotilting complex inducing the Cohen--Macaulay heart $\Hcal_\mathrm{CM}^\ed$. Assume that we have an injective cogenerator $G \in \fp{\Hcal_\mathrm{CM}^\ed}$. As above, this implies $G \in \D^\bdd_\fg(R)$ and so $G$ is a Cohen--Macaulay complex by \cref{heart-cm-complexes}. Fix $\pp \in \Spec R$. Using \cref{fpinj-local}, we easily infer that $G_\pp$ is injective in $\fp{\Hcal_\mathrm{CM}^{\ed_\pp}}$ for each $\pp \in \Spec R$. Therefore, we reduce the claim to $R$ a local ring with maximal ideal $\pp$. In view of \cref{fp-inj-cogen}, $G$ belongs to $\Def(C)$, where $C$ is the cotilting complex corresponding to the codimension function $\ed$. It follows that $\id_{R}G \leq \ed(\pp)$, which in turn implies $\id_{R}\RGamma_{\!\pp} G \leq \ed(\pp)$. Since we already know that $G$ is a Cohen--Macaulay complex, $\RGamma_{\!\pp} G$ is quasi-isomorphic to a stalk complex of an injective $R$-module in degree $\ed(\pp)$. Thus, $G$ is a Gorenstein complex with respect to $\ed$.
\end{proof}

\begin{prop}\label{Gor-duality}%
 Let $G$ be a Gorenstein complex in $\D^\bdd_\fg(R)$ and denote $S = \End_{\D(R)}(G)^\op$. Then there is a triangle equivalence $\D^\bdd_\fg(R) \cong \D^\bdd(\mod S)^\op$ induced by $\RHom_R(-,G)$.
\end{prop}
\begin{proof}
 Let $A = \dgEnd_R(G)^\op$ so that $A$ is quasi-isomorphic to $S^\op$ and $G$ is an $A$-$R$-dg-bimodule. There is the dual adjunction 
\[
\xymatrix{%
	**[l]\D^\bdd_\fg(R) \ar@/_10pt/[r]_-{\RHom_R(-,G)}  &
		**[r]\D^\bdd(\rdgMod A)^\op_{\rmod S} \ar@/_10pt/[l]_-{\RHom_A(-,G)}
},
\]
 where $\D^\bdd(\dgMod A)_{\mod S}$ is the full subcategory of $\D(\dgMod A)$ consisting of dg-modules $Z$ such that $\bigoplus_{i \in \Zbb}H^i(Z)$ is a finitely generated $S$-module.
Composing this with the natural equivalence $\epsilon\colon\D(\dgMod A) \cong \D(S)$, which restricts to an equivalence $\epsilon\colon\D^\bdd(\dgMod A)_{\mod S} \cong \D^\bdd(\mod S)$, we get a dual adjunction:
\[
\xymatrix{%
	**[l]\D^\bdd_\fg(R) \ar@/_10pt/[r]_-{\epsilon\RHom_R(-,G)}  &
		**[r]\D^\bdd(\mod S)^\op \ar@/_10pt/[l]_-{\RHom_A(\epsilon^{-1}(-),G)}
}.
\]
Let us denote the unit and counit of the latter dual adjunction as 
\begin{align*}
	\eta_X\colon X &\longrightarrow
		\RHom_A(\epsilon^{-1}\epsilon\RHom_R(X,G),G) \cong
		\RHom_S(\RHom_R(X,\epsilon G), \epsilon G), \cr
	\nu_Z\colon Z &\longrightarrow
		\epsilon\RHom_R(\RHom_A(\epsilon^{-1}Z,G),G) \cong
		\RHom_R(\RHom_S(Z,\epsilon G),\epsilon G),
\end{align*}
where $X \in \D^\bdd_\fg(R)$ and $Z \in \D^\bdd(\mod S)$; for the isomorphisms we use \cite[Theorem~12.7.2]{Yek}, note that $G \cong \epsilon G$ as objects of $\D^\bdd_\fg(R)$. Our goal is to show that both $\eta_X$ and $\nu_Z$ are quasi-isomorphisms.

A standard argument shows that the above setting is compatible with localization at a prime ideal. Since quasi-isomorphisms are detected locally on maximal ideals, we can without loss of generality assume that $R$ is a local ring and $\ed$ is the standard codimension function. Applying $-\otimes_R\what{R}$, to $\eta_X$ and $\nu_Z$, we obtain the unit and counit map
\begin{align*}
	\eta_{\what{X}}\colon \what{X} &\longrightarrow
		\RHom_{M_{k}(\what{R})}(\RHom_{\what{R}}(\what{X},
		\what{\epsilon G}), \what{\epsilon G}), \cr
	\nu_{\what{Z}}: \what{Z} &\longrightarrow
		\RHom_{\what{R}}(\RHom_{M_{k}(\what{R})}(\what{Z},
		\what{\epsilon G}),\what{\epsilon G}),
\end{align*}
using that fact that all the objects considered belong to $\D^\bdd_\fg(R)$ and \cref{Gor-prop}. We have $\what{\epsilon G} \cong D_{\what{R}}^{k}$ by \cref{Gor-prop}, and this isomorphism lives both in $\D(\what{R})$ and $\D(\what{S}) \cong \D(M_{k}(\what{R}))$. Therefore, $\eta_{\what{X}}$ and $\nu_{\what{Z}}$ are the unit and counit morphisms of the dual adjunction
\[
\xymatrix{%
	**[l]\D^\bdd_\fg(R) \ar@/_10pt/[r]_-{\RHom_R(-,D^k_{\what R})}  &
		**[r]\D^\bdd(\mod M_{k}(\what R))^\op
		\ar@/_10pt/[l]_-{\RHom_{M_{k}(\what R)}(-,D^k_{\what R})}
},
\]
for objects $\what{X} \in \D^\bdd_\fg(R)$ and $\what{Z} \in \D^\bdd(\mod {M_{k}(\what{R}}))$. But this latter adjunction arises as the duality on $\D^\bdd_\fg(R)$ induced by $D_{\what{R}}$ composed with the Morita equivalence $\Mod {\what{R}} \cong \Mod {M_{k}(\what{R}})$.
\end{proof}

The following is a characterization of the existence of a Gorenstein complex in terms of the existence of a generalized duality $\D^\bdd_\fg(R) \cong \D^\bdd(\mod S)^\op$ to a category $\mod S$ of finitely presented right modules over a ring $S$.

\begin{theorem}\label{Gor-eq}%
Let $R$ be a commutative noetherian ring of finite Krull dimension. The following are equivalent:

\begin{enumerate}[(a)]
\item there is a Gorenstein complex in $\D^\bdd_\fg(R)$;

\item $R$ is a homomorphic image of a Cohen--Macaulay ring and there is a ring $S$ such that $\fp{\Hcal_\mathrm{CM}}^{\op} \cong \mod S$.
\end{enumerate}

\noindent In addition, the ring $S$ of \textup{(b)} is an Azumaya algebra over $R$.
\end{theorem}
\begin{proof}\leavevmode
\subparagraph{$\mathrm{(b)}\Rightarrow\mathrm{(a)}$} This follows by \cref{Gorenstein-fpinj} by noting that $\mod S$ has a projective generator $S$, which in turn implies that $\fp{\Hcal_\mathrm{CM}}$ has an injective cogenerator.

\subparagraph{$\mathrm{(a)}\Rightarrow\mathrm{(b)}$} By \cref{Gor-duality}, we have the equivalence $\D^\bdd_\fg(R) \cong \D^\bdd((\mod S)^\op)$, where $S = \End_{\D(R)}(G)^\op$. Then $S$ is a module-finite $R$-algebra, in fact, it is an Azumaya $R$-algebra by \cite[Theorem~6.3.8]{NS}. Then $S$ is noetherian on both sides, and so $(\mod S)^\op$ is an artinian category. It follows that $R$ is a homomorphic image of a Cohen--Macaulay ring by \cref{char-eq}, which also shows that $(\mod S)^\op \cong \fp{\Hcal_\mathrm{CM}}$.
\end{proof}

\begin{cor}
Let $R$ be a commutative noetherian ring of finite Krull dimension. The following are equivalent:

\begin{enumerate}[(a)]
\item there is a dualizing complex over $R$;

\item there is a triangle equivalence $\D^\bdd_\fg(R) \cong \D^\bdd_\fg(R)^{\op}$;

\item $R$ is a homomorphic image of a Cohen--Macaulay ring, and $\fp {\Hcal_\mathrm{CM}} \cong (\mod R)^\op$.
\end{enumerate}
\end{cor}
\begin{proof}\leavevmode
\subparagraph{$\mathrm{(a)}\Rightarrow\mathrm{(b)}$} By definition.

\subparagraph{$\mathrm{(b)}\Rightarrow\mathrm{(c)}$} The assumption yields a derived equivalence $\D^\bdd_\fg(R) \cong \D^\bdd((\mod R)^{\op})$. Since $(\mod R)^{\op}$ is artinian, \cref{char-eq} implies (c).

\subparagraph{$\mathrm{(c)}\Rightarrow\mathrm{(a)}$} As in \cref{Gor-eq}, we see that $R$ admits a Gorenstein complex $D$ such that $\End_{\D(R)}(D) \cong R$. Then $D$ is a dualizing complex by \cref{Gor-duality}.
\end{proof}

It is well-known that a dualizing complex is a cotilting object in the category $\D^\bdd_\fg(R)$, see \cite[Remark~7.7]{HNS} for a discussion. We conclude by extending this to a tilting theoretic characterization of Gorenstein complexes.

\begin{prop}
Let $G \in \D^\bdd_\fg(R)$. Then $G$ is a Gorenstein complex if and only if it is a cotilting object in $\D^\bdd_\fg(R)$.
\end{prop}
\begin{proof}
 Let $G$ be a Gorenstein complex, then by \cref{Gor-duality} we have a triangle equivalence $\D^\bdd_\fg(R) \toeq \D^\bdd(\mod S)^\op$ induced by $\RHom_R(-,G)$ and $\RHom_A(-,G)$ where $S = \End_{\D(R)}(G)^\op$ and $A$ is its dg-resolution. Since $S$ is a tilting object in $\D^\bdd(\mod S)$, the equivalence implies that $G \cong \RHom_A(A,G)$ is a cotilting object in $\D^\bdd_\fg(R)$.

For the converse, let $G \in \D^\bdd_\fg(R)$ be a cotilting object. By definition, we have a t-structure $(\Ucal_0,\Vcal_0) = (\Perp{\leq 0}G,\Perp{>0}G)$ in $\D^\bdd_\fg(R)$, and using \cref{restrictable-props} this t-structure extends to a restrictable intermediate t-structure $(\Ucal,\Vcal)$ in $\D(R)$. By \cref{restrictable-props}, the heart $\Hcal$ of $(\Ucal,\Vcal)$ is locally coherent and also this t-structure is cotilting. We thus have a triangle equivalence $\D(R) \toeq \D(\Hcal)$, which restricts to a triangle equivalence $\D^\bdd_\fg(R) \toeq \D^\bdd(\fp{\Hcal})$ (\cite[Lemma~3.13]{HP}). Since $G$ is cotilting in $\D^\bdd_\fg(R)$, $G$ is an injective cogenerator in $\fp{\Hcal}$ by \cite[Proposition~4.3]{PV18}, and so $\fp{\Hcal}^\op \cong \mod{S}$, where $S = \End_{\D(R)}(G)^\op$. Since $G \in D^\bdd_\fg(R)$, $S$ is a module-finite $R$-algebra, and thus a right noetherian ring in particular. It follows that $\fp{\Hcal}$ is an artinian abelian category, and thus $R$ is a homomorphic image of a finite-dimensional Cohen--Macaulay ring \cref{char-eq}. The proof of \cref{char-eq} shows that $(\Ucal,\Vcal)$ is induced by a codimension function on $\Spec R$, and so $\Hcal$ is the Cohen--Macaulay heart. Then \cref{Gorenstein-fpinj} shows that $G$ is a Gorenstein complex.
\end{proof}

\bibliographystyle{amsalpha}
\bibliography{bibitems}
\end{document}